\documentclass{aptpub}

\usepackage{graphicx}
\usepackage{color}

\authornames{P.R. Jelenkovi\'c and M. Olvera-Cravioto} 
\shorttitle{Information ranking and power laws on trees} 

\def\qed{\hfill{$\Box$}}

\numberwithin{equation}{section}

\begin{document}

\title{Information Ranking and Power Laws on Trees}

\authorone[Columbia University]{Predrag R. Jelenkovi\'c}

\addressone{Department of Electrical Engineering, Columbia University, New York, NY 10027 } 

\authortwo[Columbia University]{Mariana Olvera-Cravioto} 

\addresstwo{Department of Industrial Engineering and Operations Research, Columbia University, New York, NY 10027 } 

\begin{abstract}
We consider the stochastic analysis of information ranking algorithms of large interconnected data sets, e.g. Google's PageRank algorithm for ranking pages on the World Wide Web. The stochastic formulation of the problem results in an equation of the form
$$R \stackrel{\mathcal{D}}{=} Q + \sum_{i=1}^N C_i R_i,$$
where $N, Q, \{R_i\}_{i\ge 1}, \{C,C_i\}_{i \geq 1}$ are independent non-negative random variables, $\{C,C_i\}_{i \geq 1}$ are identically distributed, and $\{R_i\}_{ i \geq 1}$ are independent copies of $R$; $\stackrel{\mathcal{D}}{=}$ stands for equality in distribution.  We study the asymptotic properties of the distribution of $R$ that, in the context of PageRank, represents the frequencies of highly ranked pages. The preceding equation is interesting in its own right since it belongs to a more general class of weighted branching processes that have been found useful in the analysis of many other algorithms. 

Our first main result shows that if $E N E[C^\alpha]=1, \alpha >0$ and $Q, N$ satisfy additional moment conditions, then $R$ has a power law distribution of index $\alpha$. This result is obtained using a new approach based on an extension of Goldie's (1991) implicit renewal theorem.  Furthermore, when $N$ is regularly varying of index $\alpha>1$, $E N E[C^\alpha]<1$ and $Q, C$ have higher moments than $\alpha$, then the distributions of $R$ and $N$ are tail equivalent. The latter result is derived via a novel sample path large deviation method for recursive random sums. Similarly, we characterize the situation when the distribution of $R$ is determined by the tail of $Q$. The preceding approaches may be of independent interest, as they can be used for analyzing other functionals on trees. We also briefly discuss the engineering implications of our results. 
\end{abstract}


         

\section{Introduction}

We consider a problem of ranking large interconnected information (data) sets,
e.g., ranking pages on the World Wide Web (Web). A solution to the preceding problem is given by Google's PageRank algorithm, the details of which are presented in Section~\ref{ss:pr}.
Given the large scale of these information sets, we adopt a stochastic approach 
to the page ranking problem, e.g. Google's PageRank algorithm. 
The stochastic formulation naturally results in an equation of the form
\begin{equation} \label{eq:GeneralPR}
R \stackrel{\mathcal{D}}{=} Q + \sum_{i=1}^N C_i R_i ,
\end{equation}
where $N, Q, \{R_i\}_{i\ge 1}, \{C,C_i\}_{i \geq 1}$ are independent non-negative random variables, $P(Q > 0) > 0$, $\{C,C_i\}_{i \geq 1}$ are identically distributed, and $\{R_i\}_{ i \geq 1}$ 
are independent copies of $R$; $\stackrel{\mathcal{D}}{=}$ stands for equality in distribution.  
We study the asymptotic properties of the distribution of $R$
that, in the context of PageRank, represents the frequencies of highly ranked pages. 
In somewhat smaller generality, the preceding stochastic setup was first 
introduced and analyzed in \cite{Volk_Litv_Dona_07} for the PageRank algorithm; the formulation given in \eqref{eq:GeneralPR} was later studied in \cite{Volk_Litv_08}.

The canonical representation given by recursion \eqref{eq:GeneralPR} is also of independent interest 
since it belongs to a more general class of weighted branching processes (WBPs)
\cite{Rosler_93, Liu_98, Kuhl_04}; the connection to WBPs is discussed in more detail in Section~\ref{SS.RelatedProcesses}. With a small abuse of notation, we also refer to our more restrictive processes as WBPs. These processes have been found useful in the average-case analysis of many algorithms
\cite{Ros_Rus_01}, e.g. quicksort algorithm \cite{Fill_Jan_01}, and thus, 
our study of recursion (\ref{eq:GeneralPR}) may be useful in these types of applications. 
Furthermore,  when $Q=1, C_i \equiv 1$, the steady state solution to (\ref{eq:GeneralPR}) 
represents the total number of individuals born in an ordinary branching process. 
Also, by letting $N$ be a Poisson random variable and 
fixing $Q = 1, C_i \equiv 1$, equation (\ref{eq:GeneralPR}) reduces to the recursion that is satisfied by the 
busy period of an M/G/1 queue. Similarly, selecting $N=1$ yields the fixed point equation satisfied  by the first order autoregressive process; see Section~\ref{SS.RelatedProcesses} for a more thorough discussion on related processes.

In Section~\ref{S.ModelDescription} we connect the iterations of recursion \eqref{eq:GeneralPR} 
to an explicit construction of a WBP on a tree, such that the sum of all the weights of the first $n$ generations of the tree are directly related to the $n$th iteration of the recursion. 
Then, in Section \ref{S.Moments} we present explicit estimates for the 
moments of the total weight, $W_n$, of the $n$th generation in the corresponding WBP.
Using these moment estimates and the WBP representation, we show in Section~\ref{SS.Convergence} that under mild conditions the iterations of \eqref{eq:GeneralPR} converge in distribution to a unique and finite steady state random variable $R$.  Hence, under the stated assumptions, this limiting distribution $P(R\le x)$ is the unique solution to (\ref{eq:GeneralPR}).  The steady state variable $R$ represents the sum of all the weights in the corresponding branching tree. 

Studying the asymptotic tail properties of the constructed steady state solution $R$ to (\ref{eq:GeneralPR}) 
represents the main focus of this paper. In particular, we study the possible causes that can result in power 
tail asymptotics for $P(R > x)$. We discover that the tail behavior of $R$ can be determined/dominated 
by the statistical properties of any of the three variables $C, N$ and $Q$. The corresponding results are presented in Sections~ \ref{S.C_dominates}, \ref{S.NDominates} and \ref{S.QDominates}, respectively.
Our emphasis on power law asymptotics is motivated by the well established empirical fact that the number 
of pages that point to a specific page (in-degree) on the Web, represented by $N$ in recursion (\ref{eq:GeneralPR}), 
follows a power law distribution; other complex data sets, e.g. citations, 
are found to posses similar power law properties as well. 

Our first main result on the tail behavior of  $P(R>x)$  is presented in Theorem~\ref{T.GoldieApplication},
showing that if $E N E[C^\alpha]=1, \alpha >0$ and $Q, N$ satisfy additional moment conditions, then $R$ has a power law distribution of index $\alpha$, 
with an explicitly characterized constant of proportionality. In particular, when $\alpha$ is an integer,
the constant of proportionality of the power law distribution is explicitly computable, see Corollary~\ref{C.explicit}.  This result is obtained by an extension of Goldie's (1991) implicit renewal theorem that we present in Theorem~\ref{T.Goldie}. This extension may be of independent interest since $R$ and $C$ in the statement of Theorem~\ref{T.Goldie} can be any two independent random variables that may satisfy a different recursion.  In the context of the broader literature on WBPs, our
results are related to the studies in \cite{Rosler_93} (see Theorem 6), and more recently in \cite{Alsm_Rosl_05}, both of which study recursion \eqref{eq:GeneralPR} using stable law methods when $Q$, $\{C_i\}$ are deterministic constants. However, these deterministic assumptions fall outside of the scope of this paper; for more details see the discussion in Section \ref{SS.RelatedProcesses} and the remarks after Theorem \ref{T.GoldieApplication}. Outside of these results, the majority of the work on WBPs considers the homogeneous equation ($Q \equiv 0$), e.g. in \cite{Liu_98} the behavior of the distribution of $R$ was characterized using stable-law distributions for $0 < \alpha \leq 1$.  Also, related results for the homogeneous case ($Q \equiv 0$) and $\alpha > 1$ can be found in Theorem 2.2 of \cite{Liu_00} and Proposition 7 of \cite{Iksanov_04}.  Interestingly, our approach for the nonhomogeneous case ($P(Q>0) > 0$) shows that the distribution of $R$ can have a uniform treatment for any $\alpha>0$. For additional comments on results related to our Theorem~\ref{T.GoldieApplication} see the remarks following its statement. Furthermore, this result may provide a new explanation of why power laws are so commonly found in the distribution of wealth since weighted branching processes appear to be reasonable models for the total wealth of a family tree.

Section \ref{S.NDominates} studies the case when $N$ is power law and dominates the tail behavior of $R$.  This is the case that more closely relates to the original formulation of PageRank and the structure of the Web graph since the in-degree $N$ is well accepted to be a power law.  Our main result in this case, stated in Theorem~\ref{T.Main_N}, shows that, when $N$ is regularly varying of index $\alpha>1$, $E N E[C^\alpha]<1$  and $Q, C$ have higher moments than $\alpha$, then the distribution of $R$ is tail equivalent to that of $N$.  Our approach in deriving this result is based on a new sample path heavy-tailed large deviation method for weighted recursions on trees. The key technical result is given by 
Proposition~\ref{P.UniformBound} that provides a uniform bound (in $n$ and $x$) on the distribution of 
the total weight of the $n$th generation $P(W_n>x)$. We would also like to point out that Proposition~\ref{P.UniformBound} resembles to some extent a classical result by Kesten (see 
Lemma~7 on p.~149 of \cite{Ath_McD_Ney_78}), which provides a uniform bound for the sum of 
heavy-tailed (subexponential) random variables. The main difference between the latter result and our uniform bound is that $n$ refers to the depth of the recursion in our case, while in Lemma~7 of \cite{Ath_McD_Ney_78}, $n$ is the number of terms in the sum. This makes the derivation of Proposition~\ref{P.UniformBound} considerably more complicated, and perhaps implausible, if it were not for the fact that we restrict our attention to regularly varying distributions, as opposed to the general subexponential class.

Section \ref{S.QDominates} investigates a third possible source of heavy tails for $R$, the one that arises from the innovation, $Q$, being power law; see Theorem~\ref{T.MainQ}. For $N=1$, this result is consistent 
with a corresponding result for the first order autoregressive process in Lemma A.3 of \cite{Mik_Sam_00}.  The proofs of more technical results are postponed to Section~\ref{S.Proofs}.

Finally, from a mathematical perspective, we would like to emphasize that our sample path large 
deviation approach as well as the extension of the implicit renewal theory, provide 
a new set of tools that can be of potential use in other applications, as well as in studying the broader class of recursions on trees, e.g., one can readily characterize the asymptotic behavior of the distribution that solves $R = Q + \max_{1\leq i\leq N} C_i R_i$. Furthermore, from an engineering perspective, our Theorem~\ref{T.Main_N} shows that for highly ranked pages, the PageRank algorithm basically reflects the popularity vote given by the number of references $N$, implying that overly inflated referencing may be advantageous.  A more detailed discussion on the engineering implications of the performance and design of ranking algorithms, e.g. PageRank, can be found at the end of Section~\ref{S.NDominates}.

\subsection{Google's algorithm: PageRank}
\label{ss:pr}

PageRank is an algorithm trademarked by Google, the Web search engine, to assign to each page a numerical weight that measures its relative importance with respect to other pages. We think of the Web as a very large interconnected graph where nodes correspond to pages.  The Google trademarked algorithm PageRank defines the page rank as: 
\begin{equation} \label{eq:PR}
R(p_i) = \frac{1-d}{n} + d \sum_{p_j \in M(p_i)} \frac{R(p_j)}{L(p_j)},
\end{equation}
where, using Google's notation, $p_1, p_2,\dots, p_n$ are the pages under consideration, $M(p_i)$ is the set of pages that link to $p_i$, $L(p_j)$ is the number of outbound links on page $p_j$, $n$ is the total number of pages on the Web, and $d$ is a damping factor, usually $d=0.85$. As noted in the original paper by Brin and Page (1998) \cite{Bri_Pag_98} PageRank ``can be calculated using a simple iterative algorithm, and corresponds to the principal eigenvector of the normalized link matrix of the Web. Also, a PageRank for 26 million web pages can be computed in a few hours on a medium size workstation."  
Other link-based ranking algorithms for web pages include the HITS algorithm, developed by Kleinberg \cite{Klein_98}, and the TrustRank algorithm \cite{Gyo_Gar_Ped_04}.

While in principle the solution to \eqref{eq:PR} reduces to the solution of a large system (possibly billions)  of linear equations, we believe that finding page ranks in such a way is unlikely to be insightful. Specifically, if one obtains the principal eigenvector of the normalized link matrix, it is hard to obtain from the solution  qualitative insights about the relationship between highly ranked pages and the in-degree/out-degree statistical properties of the graph. 

In particular, the division by the out-degree, $L(p_j)$ in equation \eqref{eq:PR}, was meant to decrease 
the contribution of pages with highly inflated referencing, i.e., those pages that basically point/reference 
possibly indiscriminately to other documents. However, the stochastic approach (to be described in the following sections) reveals that highly ranked pages are essentially insensitive to the parameters of the out-degree distribution, implying that the PageRank algorithm may not reduce the effects of overly inflated referencing (citations, voting) as originally intended, i.e., it may lead to possibly unjustifiable highly ranked pages.  An analytical explanation as to why the tail of the rank distribution is dominated by $N$ was first given in \cite{Volk_Litv_Dona_07} and \cite{Volk_Litv_08}. More discussions on this topic are provided at the end of Section~\ref{S.NDominates}. 

A stochastic approach to analyze \eqref{eq:PR} is to consider the recursion 
\begin{equation} \label{eq:StochPR}
R \stackrel{\mathcal{D}}{=} \gamma + c \sum_{i=1}^N \frac{R_i}{D_i} ,
\end{equation}
where $\gamma, c > 0$ are constants, $cE[1/D] < 1$, $N$ is a random variable independent of the $R_i$'s and $D_i$'s, the $D_i$'s are iid random variables satisfying $D_i \geq 1$, and the $R_i$'s are iid random variables having the same distribution as $R$. In terms of recursion \eqref{eq:PR}, $R$ is the rank of a random page, $N$ corresponds to the in-degree of that node, the $R_i$'s are the ranks of the pages pointing to it, and the $D_i$'s correspond to the out-degrees of each of these pages.  The experimental justification of these independence assumptions can be found in \cite{Volkovich2009}. This stochastic setup was first introduced in \cite{Volk_Litv_Dona_07}, where the process resulting after a finite number of iterations of \eqref{eq:StochPR} was analyzed. More recently, in a follow up paper \cite{Volk_Litv_08}, the more general recursion 
$$
R \stackrel{\mathcal{D}}{=} Q + \sum_{i=1}^N C_i R_i
$$
was analyzed via tauberian theorems for the cases when $N$ or $Q$ dominate. In \cite{Volk_Litv_08}, dependancy between $N$ and $Q$ is allowed, but additional moment conditions are imposed.  Recall that in the setup considered here $N, Q, \{R_i\}_{i\ge 1}, \{C,C_i\}_{i \geq 1}$ are independent non-negative random variables, $P(Q > 0) > 0$, $\{C,C_i\}_{i \geq 1}$ are identically distributed, and $\{R_i\}_{ i \geq 1}$ are independent copies of $R$.

\section{Model Description} \label{S.ModelDescription}

As outlined above, we study the sequence of random variables that are obtained by iterating \eqref{eq:GeneralPR}. Specifically, we consider for $n \geq 0$
\begin{equation} \label{eq:Rstar_Rec}
R_{n+1}^* = Q_{n} + \sum_{i=1}^{N_n} C_{n,i} R_{n,i}^*,
\end{equation}
where $\{ R_{n,i}^* \}_{i \geq 1}$ are iid copies of $R_n^*$ from the previous iteration, and $\{N_n\}, \{ C_{n,i} \}, \{ Q_{n}\}$ are mutually independent iid sequences of random variables; for $n = 0$, $R_{0,i}^*$ are iid copies of the initial value $R_0^*$.  

In this section we will discuss the weak convergence of $R_n^*$ to a finite random variable $R$, independently of the initial condition $R_0^*$. In other words, $R$ is the unique solution to \eqref{eq:GeneralPR} under the assumptions of Lemma~\ref{L.Convergence}. In particular, we will construct a process $R^{(n)}$ on a tree that converges a.s. to $R$. These convergence results may be of practical interest as well since ranking algorithms are implemented recursively. The actual proofs are postponed until Section~\ref{SS.Convergence}.  

\subsection{Construction of $R$ on a Tree} \label{SS.TreeConstruction}

To better understand the dynamics of our recursion, we give below a sample path construction of the random variable $R$ on a tree. First we construct a random tree $T$. We use the notation $\emptyset$ to denote the root node of $T$, and $A_n$, $n \geq 0$, to denote the set of all individuals in the $n$th generation of $T$, $A_0 = \{\emptyset\}$. Let $Z_n$ be the number of individuals in the $n$th generation, that is, $Z_n = |A_n|$, where $| \cdot |$ denotes the cardinality of a set; in particular, $Z_0 = 1$. We iteratively construct the tree as follows. Let $N^{(0)}$ be the number of individuals born to the root node $\emptyset$ and let $N^{(0)}, \{N^{(n)}_{i_1,\dots, i_n} \}_{n \geq 1}$ be iid copies of $N$. Define now 
$$A_1 = \{ i: 1 \leq i \leq N^{(0)} \}, \quad A_n = \{ (i_1, i_2, \dots, i_n): (i_1, \dots, i_{n-1}) \in A_{n-1}, 1 \leq i_n \leq N^{(n-1)}_{i_1, \dots, i_{n-1}} \},$$
and then the number of individuals $Z_n = |A_n|$ in the $n$th generation, $n \geq 1$, satisfies the following branching recursion 
$$Z_{n} = \sum_{(i_1, \dots, i_{n-1}) \in A_{n-1}} N^{(n-1)}_{i_1,\dots, i_{n-1}}.$$ 

Suppose now that individual $(i_1,\dots,i_n)$ in the tree has a weight ${\bf C}_{i_1,\dots,i_n}^{(n)}$ defined via the recursion
$$ {\bf C}_{i_1}^{(1)} =C_{i_1}^{(1)}, \qquad {\bf C}_{i_1,\dots, i_n}^{(n)} = C_{i_1,\dots, i_{n-1}, i_n}^{(n)} {\bf C}_{i_1,\dots,i_{n-1}}^{(n-1)}, \quad n \geq 2,$$
where ${\bf C}^{(0)} =1$ is the weight of the root node and the random variables $\{C_{i_1,\dots, i_n}^{(n)}: n \geq 0, i_k \geq 1\}$ are iid with the same distribution as $C$. Note that ${\bf C}_{i_1,\dots, i_n}^{(n)}$ is equal to the product of all the weights $C_\cdot^{(\cdot)}$ along the branch leading to node $(i_1, \dots, i_n)$, as depicted on the figure below.  Define now the process
$$W_n =  \sum_{(i_1,\dots, i_n) \in A_n} Q_{i_1,\dots, i_n}^{(n)} {\bf C}_{i_1, \dots, i_n}^{(n)}, \qquad n \geq 0,$$
where $A_n$ is the set of all individuals in the $n$th generation and $\{Q_{i_1,\dots,i_n}^{(n)}\}_{n \geq 0}$ is a 
sequence of iid random variables having the same distribution as $Q$ (see Figure \ref{F.Tree}), and independent of ${\bf C}_{\cdot}^{(\cdot)}$. 

\begin{center}
\begin{figure}[h]
\begin{picture}(430,160)(0,0)
\put(0,0){\includegraphics[scale = 0.8, bb = 0 510 500 700, clip]{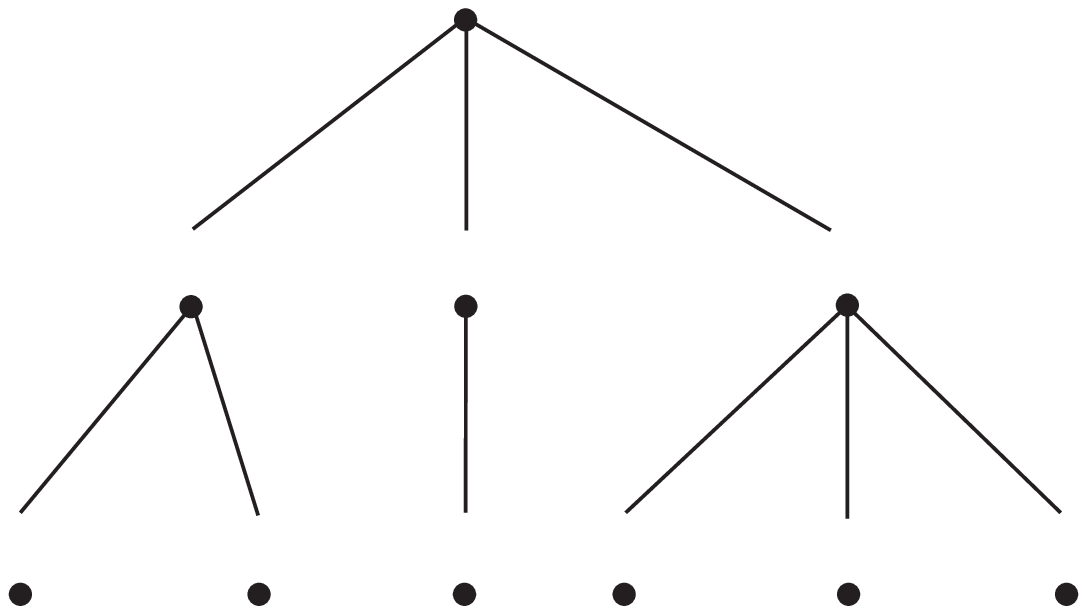}}
\put(136,150){\small ${\bf C}^{(0)}$}
\put(65,83){\small ${\bf C}_{1}^{(1)}$}
\put(127,83){\small ${\bf C}_{2}^{(1)}$}
\put(215,83){\small ${\bf C}_{3}^{(1)}$}
\put(22,17){\small ${\bf C}_{1,1}^{(2)}$}
\put(78,17){\small ${\bf C}_{1,2}^{(2)}$}
\put(126,17){\small ${\bf C}_{2,1}^{(2)}$}
\put(162,17){\small ${\bf C}_{3,1}^{(2)}$}
\put(213,17){\small ${\bf C}_{3,2}^{(2)}$}
\put(268,17){\small ${\bf C}_{3,3}^{(2)}$}
\put(350,150){\small $W_0 = Q^{(0)} {\bf C}^{(0)}$}
\put(350,135){\small $Z_0 = 1$}
\put(350,83){\small $W_1 = \sum_{i} Q_{i}^{(1)} {\bf C}_{i}^{(1)}$}
\put(350,68){\small $Z_1 = 3$}
\put(350,17){\small $W_2 = \sum_{i,j} Q_{i,j}^{(2)} {\bf C}_{i,j}^{(2)}$}
\put(350,2){\small $Z_2 = 6$}
\end{picture}
\caption{Construction on a tree}\label{F.Tree}
\end{figure}
\end{center}

Observe that when $C_\cdot^{(\cdot)} \equiv 1$ and $Q_\cdot^{(\cdot)} \equiv 1$, $W_n$ is equal to the number of individuals in the $n$th generation of the corresponding branching process, and in particular $Z_n = W_n$. Otherwise, $W_n$ represents the sum of the weights of all the individuals in the $n$th generation. Related processes known as weighted branching processes have been considered in the existing literature \cite{Rosler_93, Liu_98, Kuhl_04} and are discussed in more detail in Section~\ref{SS.RelatedProcesses}. With a small abuse of notation we also refer to our more restrictive processes as WBPs. 

Define the process $\{R^{(n)}\}_{n \geq 0}$ according to
$$R^{(n)} = \sum_{k=0}^n W_k , \qquad n \geq 0,$$
that is, $R^{(n)}$ is the sum of the weights of all the individuals on the tree. Clearly, when $Q_\cdot \equiv 1$ and $C_\cdot^{(\cdot)} \equiv 1$, $R^{(n)}$ is simply the number of individuals in a branching process up to the $n$th generation. We define the random variable $R$ according to
\begin{equation} \label{eq:R_Def}
R \triangleq \lim_{n \to \infty} R^{(n)} = \sum_{k=0}^\infty W_k.
\end{equation}
Furthermore, it is not hard to see that $R^{(n)}$ satisfies the recursion
\begin{equation} \label{eq:MainRec}
R^{(n)} = \sum_{j=1}^{N^{(0)}} C_{j}^{(1)} R^{(n-1)}_{j} + Q^{(0)},
\end{equation}
for $n \geq 1$, where $\{R_{j}^{(n-1)}\}$ are independent copies of $R^{(n-1)}$ corresponding to the tree starting with individual $j$ in the first generation and ending on the $n$th generation; note that $R_j^{(0)} = Q_j^{(1)}$.

Moreover, since the tree structure repeats itself after the first generation, $W_n$ satisfies
\begin{align*}
W_n &= \sum_{(i_1,\dots,i_n) \in A_n} Q_{i_1,\dots,i_n} ^{(n)} {\bf C}_{i_1,\dots,i_n}^{(n)} \\
&= \sum_{k = 1}^{N^{(0)}} C_{k}^{(1)}  \sum_{(k,\dots,i_n) \in A_n} Q_{k,\dots,i_n} ^{(n)} {\bf C}_{k,\dots,i_n}^{(2,n)} \\
&\stackrel{\mathcal{D}}{=} \sum_{k=1}^N C_k W_{n-1,k},
\end{align*}
where $N$, $C_k$, $W_{n-1,k}$ are independent of each other and of all other random variables, $ W_{n-1,k}$ 
has the same distribution as $W_{n-1}$, and 
${\bf C}_{k,\dots,i_n}^{(2,n)}=\prod_{j=2}^n C_{k, i_2,\dots, i_j}^{(j)} $, 
i.e., if $C_{k}^{(1)} >0$, then ${\bf C}_{k,\dots,i_n}^{(2,n)}={\bf C}_{k,\dots,i_n}^{(n)} / C_{k}^{(1)}$.

\subsection{Connection between $R_n^*$ and $R^{(n)}$} \label{SS.Connection}

We now connect the two processes $R_n^*$ and $R^{(n)}$, the one obtained by iterating \eqref{eq:GeneralPR} and the one obtained from the tree construction, respectively. To do this define
$$W_n(R_0^*) =  \sum_{(i_1,\dots, i_n) \in A_n} R^*_{0,(i_1,\dots, i_n)} {\bf C}_{i_1,\dots, i_n}^{(n)},$$
where $R^*_{0, (\cdot)}$ are iid copies of the initial condition $R_0^*$, independent of ${\bf C}_\cdot^{(n)}$, and the weights ${\bf C}^{(n)}_{\cdot}$ are the ones defined in Section~\ref{SS.TreeConstruction}. In words, $W_n(R_0)$ is the sum of all the weights in the $n$th generation of the tree with the coefficients $Q_{\cdot}^{(n)}$ substituted by the corresponding $R^*_{0,(\cdot)}$. We claim that
$$R_n^* \stackrel{\mathcal{D}}{=} R^{(n-1)} + W_n(R_0^*).$$
To see this note that for $n = 1$,
$$R_1^* = Q_0 + \sum_{i=1}^{N_0} C_{0,i} R_{0,i}^* \stackrel{\mathcal{D}}{=} Q^{(0)} {\bf C}^{(0)} + \sum_{i=1}^{N^{(0)}} {\bf C}_i^{(1)} R_{0,i}^* =  R^{(0)} + W_1(R_0^*) \qquad \text{(recall ${\bf C}^{(0)} = 1$)},$$
and by induction in $n$, 
\begin{align*}
R_{n+1}^* &= Q_{n} + \sum_{i=1}^{N_n} C_{n,i} R_{n,i}^* \\
&\stackrel{\mathcal{D}}{=} Q_n + \sum_{i=1}^{N_n} C_{n,i} (R^{(n-1)}_{i} + W_{n,i}(R_0^*)) \qquad \text{(by induction)} \\
&\stackrel{\mathcal{D}}{=} Q^{(0)}  + \sum_{i=1}^{N^{(0)}} C_i^{(1)} \left( R_i^{(n-1)} +  \sum_{(i,i_1,\dots,i_n) \in A_{n+1}} 
R_{0,(i,i_1,\dots,i_n)}^* {\bf C}_{i,i_1,\dots,i_n}^{(2,n+1)} \right)  \\
&= R^{(n)} + \sum_{i=1}^{N^{(0)}}  \sum_{(i,i_1,\dots,i_n) \in A_{n+1}} R_{0,(i,i_1,\dots,i_n)}^* {\bf C}_{i,i_1,\dots,i_n}^{(n+1)}  \\
&= R^{(n)} + W_{n+1}(R_0^*) ,
\end{align*} 
where $R_i^{(n-1)}$ corresponds to the process $R^{(n-1)}$ obtained from the tree starting with individual $i$ in the first generation (a descendent of the root) and ending on the $n$th generation, and $(R^{(n-1)}_i, W_{n,i}(R_0^*))$ are iid copies of $(R^{(n-1)}, W_n(R_0^*))$.  Since $R^{(n-1)} \to R$ a.s., it will follow from Slutsky's Theorem (see Theorem 1, p. 254 in \cite{ChowTeich1988}) that if $W_{n}(R_0^*) \Rightarrow 0$, then
$$R_n^* \Rightarrow R,$$
where $\Rightarrow$ denotes convergence in distribution. The proof of this convergence and that of the finiteness of $R$ are given in Section \ref{SS.Convergence}.  Understanding the asymptotic properties of the distribution of $R$, as defined by \eqref{eq:R_Def}, is the main objective of this paper.

\subsection{Related Processes}
\label{SS.RelatedProcesses}

As we mentioned earlier, the stochastic equation defined in \eqref{eq:GeneralPR} leads to the analysis of a process known in the literature as a weighted branching process (WBP). WBPs were introduced by R\"{o}sler \cite{Rosler_93} in a construction that  is more general than ours. More precisely, each individual in the tree has potentially an infinite number of offsprings, and each offspring inherits a certain (nonnegative) weight from its parent and multiplies it by a factor $T_i$, where the index $i$ refers to his birth order (i.e., a first born multiplies his inheritance by $T_1$, a second born by $T_2$, etc.). Each individual branches independently, using an independent copy of the sequence $T_1, T_2, \dots$. However, within the sequence, $T_1, T_2, \dots$ can be dependent. Only individuals whose weights are different than zero are considered to be alive. The construction we give in this paper would correspond to having
$$T_i = C_i 1_{(N \geq i)}.$$
The definition of a WBP described above leads to the following stochastic recursion for the total weight of the $n$th generation, 
\begin{equation} \label{eq:RoslerHomogeneous}
W_n \stackrel{\mathcal{D}}{=} \sum_{i=1}^\infty T_i W_{n-1,i}
\end{equation}
and a corresponding nonhomogeneous fixed point equation of the form
\begin{equation} \label{eq:RoslerNonHomogeneous}
R \stackrel{\mathcal{D}}{=} \sum_{i=1}^\infty T_i R_i + Q.
\end{equation}
In the construction given in \cite{Rosler_93}, the $\{T_i\}$ and $Q$ are allowed to be dependent as well. 

We now briefly describe some of the existing literature on WBPs, most of which considers the homogeneous equation, i.e. $Q \equiv 0$. The nonhomogeneous equation has only been studied for the special case when $Q$ and the $\{T_i\}$ are deterministic constants. In particular, Theorem 5 of \cite{Rosler_93} analyzes the solutions to \eqref{eq:RoslerNonHomogeneous} when $Q$ and the $\{T_i\}$ are nonnegative deterministic constants, which implies that $T_i \leq 1$ for all $i$ and $\sum_{i} T_i^\alpha \log T_i \leq 0$ for all $\alpha  > 0$, falling outside of the scope of this paper. More results about the solutions to \eqref{eq:RoslerNonHomogeneous} for the case when $Q$ and the $T_i$'s are real valued deterministic constants were derived in \cite{Alsm_Rosl_05}. 

Regarding the homogeneous equation, in \cite{Rosler_93}, the martingale structure of $W_n/m^n$ ($m = E[\sum_i T_i  ]$) was used to point out the existence of $W = \lim_{n\to \infty} W_n / m^n$, and it was shown that positive stable distributions with $\alpha \in (0,2)$ arise when $E\left[ \sum_i T_i^\alpha \right] = 1$ and some additional moment conditions are satisfied. Furthermore, for a detailed analysis of the case when $W$ follows a positive stable distribution $(0 < \alpha \leq 1)$ see \cite{Liu_98}. The convergence of $W_n/ m^n$ to $W$ was studied in \cite{Ros_Top_Vat_00}, and conditions for $W$ to belong to the domain of attraction of an $\alpha$-stable law $(1 < \alpha < 2)$ were given in \cite{Ros_Top_Vat_00}, along with an analysis of the rate of convergence. A generalization of the WBP described in \cite{Rosler_93} to a random environment was given in \cite{Kuhl_04}, where necessary and sufficient conditions for $W$ to be nondegenerate were derived.  The existence of moments of $W$ was studied in \cite{Alsm_Kuhl_07}.  The power law tail of $W$ for the critical case $E\left[ \sum_{i=1}^N C_i\right] = 1$ and  $\alpha > 1$ was derived in Theorem~2.2 of \cite{Liu_00} and Proposition 7 of \cite{Iksanov_04}.  For an even longer list of references to WBPs and related work see \cite{Kuhl_04} and \cite{Alsm_Rosl_05}. 

From the discussion above it is clear that the prior literature on WBPs is extensive, but we point out that 
the more specific structure of our model, given by \eqref{eq:GeneralPR}, as well as our novel analysis via implicit renewal theory, allow us to characterize the asymptotic power law behavior of the distribution of $R$ for all $\alpha > 0$ when the $\{C_i\}$ dominate the tail. In addition, we study the nonhomogeneous equation \eqref{eq:RoslerNonHomogeneous}, while the preceding work primarily focuses on the homogeneous case \eqref{eq:RoslerHomogeneous}. The case when $N$ dominates the tail, which is important for the page ranking problem, has not been considered until very recently in \cite{Volk_Litv_Dona_07} and \cite{Volk_Litv_08}. In reference to the latter work, our analysis is based on a new sample path approach, while the studies in \cite{Volk_Litv_Dona_07, Volk_Litv_08} use transforms and tauberian theorems as well as somewhat different assumptions.   We will provide more details on these connections throughout the paper in remarks after the corresponding theorems. 

From a different mathematical perspective, our model also constitutes a generalization of several important types of processes. For instance, by setting $N \equiv 1$, \eqref{eq:Rstar_Rec} reduces to an autoregressive process of order one. Also, by letting $N$ be a Poisson random variable and fixing $C_i \equiv 1$, $Q \equiv 1$, \eqref{eq:GeneralPR} becomes the recursion that the number of customers in a busy period of an M/G/1 queue satisfies.  Recursion \eqref{eq:PR} and its connection to the busy period when the weights $D_i$ are equal to a deterministic constant was exploited in \cite{Lit_Sch_Volk_07}. 

It is worth noting that probabilistic sample path approaches for the busy period ($C_i \equiv 1$, $Q \equiv 1$) were developed in \cite{Zwart_01, Jel_Mom_04, Bal_Dal_Klu_04}; the work in \cite{Zwart_01, Jel_Mom_04} is also relying on the theory of cycle maximum \cite{Asm_98}.  However, for our more general model (random $C_i$'s) it is not clear if there is a tractable way of generalizing this analysis.  Instead of pursuing the preceding directions, we develop a direct sample path large deviation analysis for recursive random sums that provides greater generality.

\section{Moments of $W_n$}
\label{S.Moments}

In this section we provide explicit estimates for the moments of the total weight, $W_n$, of the $n$th generation that will be used throughout the paper. In particular, we apply these estimates in Section \ref{SS.Convergence} to prove that $R_n^* \Rightarrow R$ where $R < \infty$ a.s. Our estimates may be of independent interest due to their explicit nature.

A simple calculation shows that provided $E[N], E[Q], E[C] < \infty$, then $E[W_n] < \infty$ and is given by
$$E[W_n] = E[N] E[C] E[W_{n-1}] = (E[N] E[C])^{n} E[W_0] = (E[N] E[C])^{n} E[Q].$$
We give below upper bounds on the general moments of $W_n$. 

Throughout the paper we will use $K$ to denote a large positive constant that may be different in different places, say $K = K/2$, $K = K^2$, etc.

\begin{lem} \label{L.MomentSmaller_1}
Suppose $E[Q^\beta] E[N] E[C^\beta]  < \infty$ for $0 < \beta \leq 1$, then
$$E[ W_n^\beta ] \leq (E[C^\beta] E[N])^{n} E[Q^\beta]$$
for all $n \geq 0$. 
\end{lem}

\begin{proof}
Simply note that
\begin{align}
E[W_n^\beta] &= E\left[ \left( \sum_{i=1}^N C_i W_{n-1,i} \right)^\beta \right] 
\end{align}
and use the well known inequality $\left( \sum_{i=1}^k y_i \right)^\beta \leq \sum_{i=1}^k y_i^\beta$ for $0 < \beta \leq 1$, $y_i \geq 0$ (see e.g., Exercise 4.2.1, p.~102, in \cite{ChowTeich1988}). 
\qed\end{proof}

\bigskip

The lemma for moments greater than one is given below. 

\begin{lem} \label{L.GeneralMoment}
Suppose $E[Q^\beta]< \infty$, $E[N^\beta] < \infty$, and $E[N] \max\{ E[C^\beta], E[C] \} < 1$ for some $\beta > 1$. Then, there exists a constant $K_\beta > 0$ such that 
\begin{equation*} 
E[ W_n^\beta ] \leq K_\beta (E[N] \max\{ E[C^\beta], E[C]\}  )^{n} 
\end{equation*}
for all $n \geq 0$. 
\end{lem}

The proof of Lemma \ref{L.GeneralMoment} is given in Section \ref{SS.Moments_Proofs}. 

\bigskip

{\sc Remark:} Recall that when $C \equiv 1$ and $Q \equiv 1$ then $E[W_n^\beta]$ is the $\beta$-moment of a subcritical branching process $Z_n$ and our result reduces to $E[Z_n^\beta] \leq K_\beta (E[N])^n$, which is in agreement with the classical results from branching processes, e.g. see Corollary 1 on p. 18 of \cite{Athreya_Ney_2004}. Moreover, from the proof of the integer $\beta$ case (given in Section \ref{SS.Moments_Proofs}), it is clear that $E[W_n^\beta]$ scales as $\rho^{\beta n}$ if $\rho^\beta > \rho_\beta$ and as $\rho_\beta^n$ if $\rho^\beta < \rho_\beta$, where $\rho = E[N] E[C]$ and $\rho_\beta = E[N] E[C^\beta]$. Note that this is not quite the same as our upper bounds, and the reason we choose the geometric term $(\rho \vee \rho_\beta)^n$ instead is that it makes the proofs simpler and is sufficient for our purposes. Similar techniques to those used in proving the preceding lemmas can yield, with some additional work, lower bounds for the $\beta$-moments of $W_n$, showing that the correct leading term is $(\rho^\beta \vee \rho_\beta)^n$.

More technical results dealing with the existence of the $\beta$-moments of $W \triangleq \lim_{n \to \infty} W_n/ \rho^n$ can be found in \cite{Alsm_Kuhl_07}.  There, necessary and sufficient conditions are given for the finiteness of $E[W^\beta L(W)]$ when $\beta \geq 1$ and $L(\cdot)$ is slowly varying (see Theorems 1.2 and 1.3). In particular, the approach the authors take is to first normalize the process so that $\rho = E[W_1] = 1$, and then impose a condition that in our case reduces to $\rho_\beta = E[N] E[C^\beta]  < 1$, that is, they preclude the situation where $W_n^\beta$ might scale as $\rho_\beta^n$ when $\rho^\beta < \rho_\beta$. An example where $E[W_n^\beta]$ scales as $\rho_\beta^n$ is when $N \equiv 1$, since then $W_n^\beta \stackrel{\mathcal{D}}{=} Q^\beta \prod_{i=1}^n C_i^\beta$. 

Furthermore, observe that when $\rho = 1$ and $\rho_\beta < 1$ for $\beta > 1$, our proof of the lemma shows that $\limsup_{n \to \infty} E[W_n^\beta ] < \infty$, but it does not converge to zero, which is in agreement with \cite{Alsm_Kuhl_07}. However, since we study $R$, it is necessary to have $\rho < 1$ for the finiteness of $E[R^\beta]$.   Otherwise, if $\rho =1$, $\rho_\beta < 1$, $\beta > 1$, then $E[R^{(n)}] = n E[Q]$ which by monotone convergence and \eqref{eq:R_Def} implies that $E[R] = \infty$, and therefore, by convexity, $E[R^\beta] = \infty$.

\subsection{Convergence of $R_n^*$ and finiteness of $R$} \label{SS.Convergence}

As discussed in Section \ref{SS.Connection}, there are two issues regarding the process $R_n^*$ that remain to be addressed. One, is the proof that 
$$R_n^* \Rightarrow R = \sum_{k=0}^\infty W_k$$
for any initial condition $R_0^*$; the other one is the finiteness of $R$. The lemma below shows that $R< \infty$ a.s. 

\begin{lem} \label{L.Stability}
Suppose that $E[Q^\beta] < \infty$, $E[N^\beta] < \infty$, and either $E[N] E[C^\beta] < 1$ for some $0 < \beta < 1$, or $E[N]\max\{ E[C], E[C^\beta]\} < 1$ for some $\beta \geq 1$. Then, $E[R^\gamma] < \infty$ for all $0  < \gamma \leq \beta$, and in particular, $R < \infty$ a.s. Moreover, if $\beta \geq 1$, $R^{(n)} \stackrel{L_\beta}{\to} R$, where $L_\beta$ stands for convergence in $(E|\cdot|^\beta)^{1/\beta}$ norm. 
\end{lem}

\begin{proof}
Let 
$$\eta = \begin{cases} E[N] E[C^\beta], & \text{ if }\beta < 1 \\  E[N] \max\{E[C], E[C^\beta] \}, & \text{ if } \beta \geq 1. \end{cases}$$ 
Then by Lemmas \ref{L.MomentSmaller_1} and \ref{L.GeneralMoment}, 
\begin{equation} \label{eq:EW_n}
E[W_n^\beta] \leq K \eta^n 
\end{equation}
for some $K > 0$. Suppose $\beta \geq 1$, then, by monotone convergence and Minkowski's inequality, 
\begin{align*}
E[R^\beta] &= E\left[ \lim_{n\to\infty} \left(\sum_{k=0}^n W_k \right)^\beta  \right] = \lim_{n\to \infty} E\left[ \left(\sum_{k=0}^n W_k\right)^\beta  \right] \\
&\leq \lim_{n\to\infty} \left( \sum_{k=0}^n E[W_k^\beta]^{1/\beta} \right)^\beta \leq K \left( \sum_{k=0}^\infty \eta^{k/\beta}   \right)^\beta < \infty.
\end{align*}
This implies that $R < \infty$ a.s. When $0 < \beta < 1$ use the inequality $\left( \sum_{k=0}^n y_k \right)^\beta \leq \sum_{k=0}^n y_k^\beta$ for any $y_i \geq 0$ instead of Minkowski's inequality. Furthermore, for any $0 < \gamma \leq \beta$, 
$$E[R^\gamma] = E\left[ (R^\beta)^{\gamma/\beta}\right] \leq \left(E[R^\beta] \right)^{\gamma/\beta} < \infty.$$
That $R^{(n)} \stackrel{L_\beta}{\to} R$ whenever $\beta \geq 1$ follows from noting that $E[|R^{(n)} - R|^\beta] = E\left[ \left( \sum_{k = n+1}^\infty W_k \right)^\beta \right]$ and applying the same arguments used above to obtain the bound $E[|R^{(n)} - R|^\beta] \leq K \eta^{n+1}/(1 - \eta^{1/\beta})^\beta$. 
\qed\end{proof}

\bigskip

Next, by monotone convergence in equation \eqref{eq:MainRec} it can be verified that $R$ must solve
$$R \stackrel{\mathcal{D}}{=} \sum_{i=1}^N C_i R_i + Q,$$
where $\{R_i\}_{i \geq 1}$ are iid copies of $R$, independent of $N$, $Q$, and $\{C_i\}$. 

We now turn our attention to the proof of the convergence of $R_n^*$ to $R$. Recall from Section \ref{SS.Connection} that
\begin{equation} \label{eq:connection}
R_n^* \stackrel{\mathcal{D}}{=} R^{(n-1)} + W_n(R_0^*),
\end{equation}
where
$$W_n(R_0^*) = \sum_{(i_1,\dots,i_n) \in A_n} R_{0,(i_1,\dots,i_n)}^* {\bf C}_{i_1,\dots,i_n}^{(n)}.$$
The following lemma shows that $R_n^* \Rightarrow R$ for any initial condition $R_0^*$ satisfying a moment assumption. 

\begin{lem} \label{L.Convergence}
For any $R_0^* \geq 0$, if $E[Q^\beta]< \infty$, $E[(R_0^*)^\beta] < \infty$ and $E[N] E[C^\beta] < 1$ for some $0 < \beta \leq 1$, then
$$R_n^* \Rightarrow R,$$
with $E[R^\beta] < \infty$. Furthermore, under these assumptions, the distribution of $R$ is the unique solution with finite $\beta$ moment to recursion \eqref{eq:GeneralPR}. 
\end{lem}

\begin{proof}
In view of \eqref{eq:connection}, and since $R^{(n)} \to R$ a.s., the result will follow from Slutsky's Theorem (see Theorem 1, p. 254 in \cite{ChowTeich1988}) once we show that $W_n(R_0^*) \Rightarrow 0$. Recall that $W_n(R_0^*)$ is the same as $W_n$ if we substitute the $Q_{i_1, \dots, i_n}$ by the $R_{0,(i_1,\dots,i_n)}^*$. Fix $\epsilon > 0$, then
\begin{align*}
P( W_n(R_0^*) > \epsilon) &\leq \epsilon^{-\beta} E[ W_n(R_0^*)^\beta] \\
&\leq \epsilon^{-\beta} (E[C^\beta] E[N])^n E[(R_0^*)^\beta] \qquad \text{(by Lemma \ref{L.MomentSmaller_1})} .
\end{align*}
Since by assumption the right hand side converges to zero as $n \to \infty$, then $R_n^* \Rightarrow R$. Furthermore, $E[R^\beta] < \infty$ by Lemma \ref{L.Stability}. Clearly, the distribution of $R$ represents the unique solution with finite $\beta$-moment to \eqref{eq:GeneralPR}, since any other possible solution would have to converge to the same limit.  
\qed\end{proof}

{\sc Remarks:} (i) Note that when $E[N] < 1$, then the branching tree is a.s. finite and no conditions on the $C$'s are necessary for $R < \infty$ a.s. This corresponds to the second condition in Theorem 1 of \cite{Brandt_86}. (ii) In view of the same theorem from \cite{Brandt_86}, one could possibly establish the convergence of $R_n^* \Rightarrow R < \infty$ under milder conditions. However, since the conditions that we will impose on $N$, $Q$ and $C$ in the main theorems will be stronger, this lemma is not restrictive. Furthermore, the initial values, $R_0^*$, are typically small (e.g. constant in applications), and thus the polynomial moment condition imposed on $R_0^*$ is general enough.

\section{The case when the $C$'s dominate: Implicit renewal theory}
\label{S.C_dominates}

In this section we study the power law phenomenon that arises from the multiplicative effects of the weights $\{C_i\}$ in \eqref{eq:GeneralPR}.  

\subsection{Implicit Renewal Theorem on Trees} \label{S.Renewal}

One observation that will help gain some intuition about \eqref{eq:MainRec} is to consider the case when $N \equiv 1$. The process $\{R^{(n)}\}$ then reduces to a (random coefficient) autoregressive process of order one, whose steady state solution satisfies
$$R \stackrel{\mathcal{D}}{=} Q + C R,$$
where $R$ is independent of $C$ and $Q$. This is precisely one of the stochastic recursions considered in \cite{Goldie_91} (see also \cite{Kesten_73}), where it is shown that under the assumption that $E[C^\alpha] = 1$ and some other technical conditions on the distribution of $C$ and $Q$, we have that
\begin{equation} \label{eq:PowerLaw}
P(R > x) \sim H x^{-\alpha}
\end{equation}
for some (computable) constant $H > 0$ (see Theorem 4.1 in \cite{Goldie_91}). The fact that the index of the power law depends on the distribution of the weights is already promising in terms of our goal of identifying other sources of power law behavior. 

Informally speaking, the recursions studied in \cite{Goldie_91} are basically multiplicative away from the boundary. However, \eqref{eq:GeneralPR} always has an additive component given by $\sum_{i=1}^N C_i R_i$ regardless of how far from the boundary one may be. Fortunately, due to the heavy-tailed nature of $R$, our intuition says that it is only one of the additive $C_iR_i$ components that determines the behavior of \eqref{eq:GeneralPR}, thus the sum will behave as the maximum term, simplifying to
\begin{equation} \label{eq:heuristic}
P\left( Q + \sum_{i=1}^N C_i R_i > x \right)  \sim E[N] P(C R > x),
\end{equation}
assuming that $Q$ has a light enough tail.  This heuristic suggests the following generalization of Theorem 2.3 from \cite{Goldie_91}. 

Here, we would like to emphasize that $R$ and $C$ in the following theorem can be any two  independent random variables that satisfy the stated conditions, i.e., they do not have to be related by recursion \eqref{eq:GeneralPR}. Hence, the theorem may be of potential use in other applications. Note that we prove the theorem for a general constant $m$, that in our application refers to $E[N]$, as suggested by \eqref{eq:heuristic}.

\begin{thm} \label{T.Goldie}
Suppose $C \geq 0$ a.s., $0 < E[C^\alpha \log C] < \infty$ for some $\alpha > 0$, and that the conditional distribution of $\log C$ given $C \neq 0$ is nonarithmetic. Suppose further that $R$ is independent of $C$, $m E[C^\alpha] = 1$, and that $E[R^\beta] < \infty$ for any $0 < \beta < \alpha$. If
\begin{equation} \label{eq:Goldie_condition}
\int_0^\infty \left| P(R > t) - m P(CR > t) \right| t^{\alpha-1} dt < \infty,
\end{equation}
then
$$P(R > t) \sim H t^{-\alpha}, \qquad t \to \infty,$$
where $H \geq 0$ is given by
$$H = \frac{1}{m E[C^\alpha \log C]} \int_{0}^\infty v^{\alpha -1} (P(R > v) - mP(C R > v)) \, dv.$$ 
\end{thm}

The proof of this theorem follows the same steps as Theorem 2.3 from \cite{Goldie_91}, and is presented in Section \ref{SS.CDominates_Proofs}. 

{\sc Remarks:} (i) As pointed out in \cite{Goldie_91}, the statement of the theorem has content only when $R$ has infinite moment of order $\alpha$, since otherwise the constant $H = (\alpha E[N] E[C^\alpha\log C])^{-1} (E[R^\alpha] - E[N] E[(CR)^\alpha])$ will be zero by independence of $R$ and $C$. (ii) Note that some of the assumptions of Theorem \ref{T.Goldie} are different than the corresponding ones from Theorem 2.3 in \cite{Goldie_91}. In particular, it is no longer the case that convexity implies $E[C^\alpha \log C] > 0$ whenever $\alpha$ solves $m E[C^\alpha] = 1$ and $E[C^\alpha \log C] < \infty$, since if $m > 1$ it is possible to construct  counterexamples, hence the need to include this as an assumption. Another difference is our requirement that $E[R^\beta] < \infty$ for any $0 < \beta < \alpha$.  In the case of applying Theorem~\ref{T.Goldie} to equation \eqref{eq:GeneralPR}, the condition on $E[R^\beta]$ is not restrictive since we readily obtain the moments of $R$ for $0<\beta <\alpha$ from the computed moments of $W_n$ from Section \ref{S.Moments}.  (iii) A similar result for the case when $\log C$ is lattice valued can be derived using the corresponding renewal theorem.  

In what follows we will use the preceding theorem to derive the asymptotic behavior of $P(R > x)$ where $R$, as given by \eqref{eq:R_Def}, satisfies \eqref{eq:GeneralPR}. Here, the main difficulty will be to show that condition \eqref{eq:Goldie_condition} holds. For brevity we use $x \vee y$ to denote $\max\{x, y\}$ and $x \wedge y$ to denote $\min\{x, y\}$. 

\bigskip

\begin{thm} \label{T.GoldieApplication}
Suppose that $0 < E[C^\alpha \log C] < \infty$ for some $\alpha > 0$, the conditional distribution of $\log C$ given $C \neq 0$ is nonarithmetic, and that $C$ and $R$ are independent, where $R$ is defined by \eqref{eq:R_Def}. Assume that $E[N] E[C^\alpha] = 1$, $0 < E[Q^{\alpha}] < \infty$ and $E[N^{\alpha \vee (1+\epsilon)}] < \infty$ for some $0 < \epsilon < 1$; if $\alpha > 1$ assume further that $E[N] E[C] < 1$. Then,
$$P(R > t) \sim H t^{-\alpha}, \qquad t \to \infty,$$
where
\begin{align*}
H &= \frac{1}{E[N] E[C^\alpha \log C]} \int_{0}^\infty v^{\alpha -1} (P(R > v) - E[N]P(C R > v)) \, dv \\
&= \frac{E\left[ \left( \sum_{i=1}^N C_i R_i +Q \right)^\alpha - \sum_{i=1}^N (C_i R_i )^\alpha \right]}{\alpha E[N]E[C^\alpha\log C]}.
\end{align*}
\end{thm}

{\sc Remarks:} (i) Note that the second expression for $H$ is more suitable for actually computing it, especially in the case of $\alpha$ being an integer, as will be stated in the forthcoming corollary.  (ii) When $\alpha$ is not an integer we can derive an explicit bound on $H$ by using the forthcoming Lemma \ref{L.Max_Approx} and \eqref{eq:Alpha_diff}. (ii) For the homogeneous equation ($Q \equiv 0$) and $\alpha > 1$, closely related results to our theorem can be found in Theorem 2.2 of \cite{Liu_00} and Proposition 7 of \cite{Iksanov_04}. The approach from \cite{Liu_00} transforms the recursion $W \stackrel{\mathcal{D}}{=} \sum_{i=1}^N C_i W_i$ for the critical case $E[W] = 1$, $E\left[ \sum_{i=1}^N C_i \right] = 1$ to a first order difference (autoregressive) equation on a different probability space, see Lemma 4.1 in \cite{Liu_00}. Note that the tail behavior of $W$ does not imply that of $R$. Furthermore, it appears that the method from \cite{Liu_00} does not extend to the nonhomogeneous case   since the proof of Lemma~4.1 in \cite{Liu_00} critically depends on having both $E[W] = 1$ and $E\left[ \sum_{i=1}^N C_i \right] = 1$, which is only possible when $Q \equiv 0$. For $0 < \alpha \leq 1$, the homogeneous equation was studied in \cite{Liu_98} using stable laws. (iii) Related results for the nonhomogeneous equation with deterministic constants $Q, \{C_i\}$, $N = \infty$, have been considered in \cite{Rosler_93} (see Theorem 5), and more recently in \cite{Alsm_Rosl_05}, also using stable laws. (iv) Moreover, the results obtained in the references cited above appear to be less explicit in the expression for $H$ than the statement of Theorem \ref{T.GoldieApplication}, as Corollary \ref{C.explicit} below illustrates.  (v) Furthermore, Theorem~ \ref{T.Goldie} and the preceding technique of Theorem~\ref{T.GoldieApplication} can be adapted to analyze other, possibly non-linear, recursions on trees, e.g., one can characterize the asymptotic behavior of $P(R > x)$ that solves
$$R = Q + \max_{1 \leq i \leq N} C_i R_i.$$

We also want to point out that one can obtain the logarithmic asymptotics of $R$, that is, the behavior of $\log P(R > x)$, much easier and under less restrictive conditions, e.g. $\log C_i$ needs not be nonarithmetic (this condition is required because of the use of the Renewal Theorem).  An upper bound can be obtained from Lemma \ref{L.Stability} and Markov's inequality. For the lower bound, using the notation from Section \ref{SS.TreeConstruction}, we obtain
\begin{align*}
P(R > x) &\geq P(W_n > x) \geq P\left( \max_{1 \leq i \leq N} C_i W_{n-1,i} > x \right) \\
&= E\left[ (1 - P(C W_{n-1} \leq x)^N) \right] \\
&\geq E\left[ N P(C W_{n-1} \leq x )^N  \right] P(C W_{n-1} > x),
\end{align*}
where in the last step we used the relation $1-t^m \geq m t^m (1-t)$ for $0 \leq t \leq 1$.  Now we use the fact that $P(C W_{n-1} \leq x) \geq P(R \leq x)$, for all $x$, to show that 
\begin{align*}
P(R > x) &\geq E\left[ N P(R \leq x )^N  \right] P(C W_{n-1} > x) \\
&\geq E\left[ N P(R \leq x )^N \right] P\left( C_1 \max_{1\leq i\leq N} C_{2,i} W_{n-2,i}  > x \right) \\
&\geq E\left[ N P(R \leq x )^N  \right] E\left[ N P(C_1 C_2 W_{n-2} \leq x )^N  \right] P(C_1 C_2 W_{n-2} > x),
\end{align*}
which, by using $P(C_1 C_2 W_{n-2} \leq x) \geq P(R \leq x)$, for all $x$, yields
\begin{align*}
P(R > x) &\geq \left(E\left[ N P(R \leq x )^N  \right] \right)^2 P(C_1 C_2 W_{n-2} > x). 
\end{align*}
Next, by continuing this inductive argument we obtain
\begin{align*}
P(R > x) &\geq \left(E\left[ N P(R \leq x )^N \right] \right)^n P\left( Q \prod_{i=1}^n C_i > x \right). 
\end{align*}
Finally, for any $0 < \epsilon < 1$, we can choose $x_0$ such that 
$E\left[ N P(R \leq x_0 )^N  \right]\geq(1-\epsilon)E [N]$, implying that 
for all $n \geq 0$ and $x \geq x_0$,
$$P(R > x) \geq (1-\epsilon)^n (E[N])^n P\left( Q \prod_{i=1}^n C_i > x \right) \geq P(Q > 1/\log x) (1-\epsilon)^n (E[N])^n P\left( \prod_{i=1}^n C_i > x \log x \right) .$$
Now define $S_n = \log C_1 + \dots + \log C_n$, $\kappa(\theta) = \log E[ C^\theta]$, and choose $\alpha$ to be the solution to $\kappa(\alpha) = -\log E[N]$ (i.e. $E[N]E[C^\alpha] = 1$). The proof can be completed by choosing $n = n(x) = \log (x\log x)/\mu_\alpha$, where $\mu_\alpha = \kappa'(\alpha) = E[C^\alpha \log C]/E[C^\alpha] > \kappa'(0) = E[\log C]$ by convexity of $\kappa(\cdot)$. Then, by Theorem 2.1 in Chapter XIII in \cite{Asm2003}, 
$$\liminf_{x \to \infty} \frac{\log P(R > x)}{\log x} \geq \frac{\log((1-\epsilon) E[N])}{\mu_\alpha} + \liminf_{x \to \infty} \frac{\log P(S_{n(x)} >  \mu_\alpha n(x))}{\mu_\alpha n(x)} = \frac{\log(1-\epsilon)}{\mu_\alpha} -\alpha.$$
Hence, one can derive with a considerably smaller effort the following theorem.

\begin{thm} 
Suppose that $0 < E[C^\alpha \log C] < \infty$ for some $\alpha > 0$, and that $R$ is given by \eqref{eq:R_Def}. Assume that $E[N] E[C^\alpha] = 1$, $0 < E[Q^{\alpha}] E[N^\alpha] < \infty$; if $\alpha > 1$ assume further that $E[N] E[C] < 1$. Then,
$$\log P(R > t) \sim -\alpha \log t, \qquad t \to \infty.$$
\end{thm}

Therefore, the majority of the work in proving Theorem \ref{T.GoldieApplication}  goes into the derivation of the exact asymptotic. Furthermore, it is worth noting that the logarithmic approach, although less precise, can be obtained in a more general setting. For example, one can have $C_\cdot^{(\cdot)}$ to be dependent across different generations, as in the so called WBP in a random environment.  Here, one could derive the asymptotics of $\log P(R > x)$ if $E\left[ \left( \prod_{i=1}^n C_{(1,1,\dots,1)}^{(n)} \right)^\alpha \right]$ satisfies the polynomial type G\"{a}rtner-Ellis conditions that were recently considered in \cite{Jel_Tan_07}.

\begin{cor} \label{C.explicit}
For integer $\alpha \geq 1$, and under the same assumptions of Theorem \ref{T.GoldieApplication}, the constant $H$ can be explicitly computed as a function of $E[R^k], E[C^k], E[Q^k]$, $0 \leq k \leq \alpha-1$. In particular, for $\alpha = 1$,
$$H = \frac{E[Q]}{E[N] E[C \log C]},$$
and for $\alpha = 2$,
$$H = \frac{E[Q^2] + 2 E[Q] E[C] E[N] E[R] +  E[N(N-1)] (E[C] E[R])^2 }{2 E[N] E[C^2 \log C]}, \qquad E[R] = \frac{E[Q]}{1-E[N]E[C]}.$$
\end{cor}

\begin{proof}
The proof follows directly from multinomial expansions of the second expression for $H$ in Theorem~\ref{T.GoldieApplication}. 
\qed\end{proof}

Before giving the proof of Theorem \ref{T.GoldieApplication} we state the following three preliminary lemmas. The proof of Lemma~\ref{L.Alpha_Moments} is given in Section \ref{SS.Moments_Proofs} and the proof of Lemma \ref{L.Max_Approx} is given in Section \ref{SS.CDominates_Proofs}. 

\bigskip

\begin{lem} \label{L.Moments_R}
Suppose that $0 < E[C^\alpha \log C] < \infty$ for some $\alpha > 0$ and $E[N]E[C^\alpha] = 1$; if $\alpha > 1$ suppose further that $E[N] E[C] < 1$. Assume also that $E[Q^\alpha] < \infty$, $E[N^{\alpha \vee 1}] < \infty$. Then,
$$E[R^\beta] < \infty$$
for all $0 < \beta < \alpha$. 
\end{lem} 

\begin{proof}
The derivative condition $0 < E[C^\alpha \log C] < \infty$ and $E[N] E[C^\alpha] = 1$ imply that $E[N] E[C^\beta] < 1$ for all $\beta < \alpha$ that are sufficiently close to $\alpha$. Hence, the conclusion of the result follows from Lemma \ref{L.Stability}. 
\qed\end{proof}

\bigskip

\begin{lem} \label{L.Alpha_Moments}
Let $\beta > 1$ and $p = \lceil \beta \rceil \in \{2, 3, 4, \dots\}$. For any sequence of nonnegative iid random variables $\{Y, Y_i\}_{i\geq 1}$ and any $k \in \{1,2,3,\dots\}$ we have
$$E\left[ \left( \sum_{i=1}^k Y_i \right)^\beta - \sum_{i=1}^k Y_i^\beta \right] \leq k^\beta E[ Y^{p-1} ]^{\beta/(p-1)}.$$
\end{lem}

\bigskip

\begin{lem} \label{L.Max_Approx}
Suppose $\{C, C_i\}$ and $\{R, R_i\}$ are iid sequences of nonnegative random variables independent of each other and of $N$. Assume that $E[C^\alpha] < \infty$, $E[N^{1 + \epsilon}] < \infty$ for some $0 < \epsilon < 1$, and $E[R^\beta]< \infty$ for any $0 < \beta < \alpha$. Then, 
$$0 \leq \int_{0}^\infty \left( E[N] P(C R > t) - P\left( \max_{1\leq i \leq N} C_i R_i > t \right) \right)  t^{\alpha -1} \, dt = \frac{1}{\alpha} E \left[  \sum_{i=1}^N  \left(C_i R_i \right)^\alpha - \left( \max_{1\leq i \leq N} C_i R_i \right)^\alpha   \right] 
 < \infty.$$
\end{lem}

\bigskip

\begin{proof}[Proof of Theorem \ref{T.GoldieApplication}]
By Lemma \ref{L.Moments_R} we know that $E[R^\beta] < \infty$ for any $0 < \beta < \alpha$. The statement of the theorem with the first expression for $H$ will follow from Theorem \ref{T.Goldie} once we prove condition \eqref{eq:Goldie_condition} for $m = E[N]$.  Define
$$R^* = \sum_{i=1}^N C_i R_i + Q.$$
Then, 
\begin{align*}
\left| P(R>t) - E[N] P(CR > t) \right|  &\leq \left| P(R > t) - P\left( \max_{1\leq i \leq N} C_i R_i > t \right) \right|   \\
&\hspace{5mm} + \left| P\left( \max_{1\leq i \leq N} C_i R_i > t \right) - E[N] P(CR > t)  \right|.
\end{align*}
Since $R \stackrel{\mathcal{D}}{=} R^* \geq \max_{1\leq i\leq N} C_i R_i$, the first absolute value disappears. For the second one note that by the union bound
\begin{align*}
E[N] P(CR > t) - P\left( \max_{1\leq i \leq N} C_i R_i > t \right) &= E\left[ N P(CR > t) - 1 + P(CR \leq t)^N \right] \geq 0.
\end{align*}
It follows that
\begin{align}
 \left| P(R > t) - E[N] P(CR > t) \right|  &\leq P(R > t) - P\left( \max_{1\leq i \leq N} C_i R_i > t \right) \notag \\
&\hspace{5mm} + E[N] P(CR > t) - P\left( \max_{1\leq i \leq N} C_i R_i > t \right)  . \notag 
\end{align}
Note that we only need to verify that
$$\int_0^\infty  \left( P(R > t) - P\left( \max_{1\leq i \leq N} C_i R_i > t \right) \right) t^{\alpha-1} \, dt  < \infty,$$
since 
$$\int_0^\infty \left( E[N]P(CR > t) - P\left( \max_{1\leq i \leq N} C_i R_i > t \right) \right) t^{\alpha-1} dt < \infty$$
by Lemma~\ref{L.Max_Approx}.  To see this note that $R \stackrel{\mathcal{D}}{=} R^*$ and $1_{(R^* > t)} - 1_{(\max_{1\leq i\leq N} C_i R_i > t)} \geq 0$, thus, by Fubini's Theorem, we have
\begin{equation} \label{eq:Fubini}
\int_0^\infty  \left( P(R > t) - P\left( \max_{1\leq i \leq N} C_i R_i > t \right) \right) t^{\alpha-1} \, dt = \frac{1}{\alpha} E \left[ (R^*)^\alpha - \left( \max_{1\leq i \leq N} C_i R_i \right)^\alpha   \right].
\end{equation}
If $0 < \alpha \leq 1$ we apply the inequality $\left( \sum_{i=1}^k x_i \right)^\beta \leq \sum_{i=1}^k x_i^\beta$ for $0 < \beta \leq 1$, $x_i \geq 0$, to obtain
$$E \left[ (R^*)^\alpha - \left( \max_{1\leq i \leq N} C_i R_i \right)^\alpha   \right] \leq E \left[ Q^\alpha + \sum_{i=1}^N (C_iR_i)^\alpha - \left( \max_{1\leq i \leq N} C_i R_i \right)^\alpha   \right],$$ 
which is finite by Lemma \ref{L.Max_Approx} and the assumption $E[Q^\alpha] < \infty$. If $\alpha > 1$ we use the well known inequality $\left(\sum_{i=1}^k x_i \right)^\alpha \geq \sum_{i=1}^k x_i^\alpha$, $x_i\geq 0$ (see Exercise 4.2.1, p. 102, in \cite{ChowTeich1988}) to split the expectation as follows
\begin{align*}
E \left[ (R^*)^\alpha - \left( \max_{1\leq i \leq N} C_i R_i \right)^\alpha   \right]  &=   E \left[ (R^*)^\alpha - \sum_{i=1}^N  \left(C_i R_i \right)^\alpha   \right]  +   E \left[  \sum_{i=1}^N  \left(C_i R_i \right)^\alpha - \left( \max_{1\leq i \leq N} C_i R_i \right)^\alpha   \right],
\end{align*}
which can be done since both expressions inside the expectations on the right hand side are nonnegative. The second expectation is again finite by Lemma \ref{L.Max_Approx}. To see that the first expectation is finite let $S = \sum_{i=1}^N  C_i R_i$ and note that $R^* = S + Q$, where $S$ and $Q$ are independent. Let $p = \lceil \alpha \rceil$ and note that $1 \leq p-1 < \alpha$. Then, by Lemma \ref{L.Alpha_Moments}, 
\begin{align*}
E \left[ (R^*)^\alpha - \sum_{i=1}^N  \left(C_i R_i \right)^\alpha   \right] &= E\left[ (S+Q)^\alpha - S^\alpha \right] + E\left[ \left( \sum_{i=1}^N C_i R_i \right)^\alpha - \sum_{i=1}^N (C_iR_i)^\alpha \right] \\
&\leq E\left[ (S+Q)^\alpha - S^\alpha \right]  + E\left[  N^\alpha \right] (E[(CR)^{p-1}])^{\alpha/(p-1)}  .
\end{align*}
The second expectation is finite since by Lemma \ref{L.Stability}, $E[R^\beta] < \infty$ for any $0 <\beta < \alpha$. For the first expectation we use the inequality 
$$(x+t)^\kappa \leq \begin{cases}
x^\kappa + t^\kappa, & 0 < \kappa \leq 1, \\
x^\kappa + \kappa (x+t)^{\kappa-1} t, & \kappa > 1,
\end{cases}$$
for any $x,t \geq 0$. We apply the second expression $p-1$ times and then the first one to obtain 
$$(x+t)^\alpha \leq x^\alpha + \alpha (x+t)^{\alpha-1} t  \leq \dots \leq x^\alpha + \sum_{i=1}^{p-2} \alpha^i  x^{\alpha-i} t^i +  \alpha^{p-1} (x+t)^{\alpha-p+1} t^{p-1}  \leq x^\alpha + \alpha^p t^\alpha + \alpha^p \sum_{i=1}^{p-1} x^{\alpha-i} t^i.$$
We conclude that
\begin{equation} \label{eq:Alpha_diff}
E\left[(S+Q)^\alpha - S^\alpha\right] \leq \alpha^p E[Q^\alpha] + \alpha^p \sum_{i=1}^{p-1} E[S^{\alpha-i}] E[Q^i], 
\end{equation}
where $E[S^{\alpha-i}] \leq E[(R^*)^{\alpha-i}] < \infty$ for any $1\leq i\leq p-1$ by Lemma \ref{L.Stability}. 

Finally, applying Theorem \ref{T.Goldie} gives 
$$P(R > t) \sim H t^{-\alpha},$$
where $H = (E[N] E[C^\alpha \log C])^{-1} \int_0^\infty v^{\alpha-1} (P(R > v) - E[N] P(CR > v)) \, dv$. 

To obtain the second expression for $H$ note that
\begin{align}
&\int_0^\infty v^{\alpha-1}  (P(R > v) - E[N] P(CR > v)) \, dv \notag \\
&= \int_0^\infty v^{\alpha-1}  \left( E\left[1_{(\sum_{i=1}^N C_i R_i + Q > v)}\right]  - E\left[ \sum_{i=1}^N 1_{(C_iR_i > v)}  \right] \right) \, dv \notag \\
&= E \left[   \int_0^\infty v^{\alpha-1}  \left(  1_{(\sum_{i=1}^N C_i R_i + Q > v)}  -  \sum_{i=1}^N 1_{(C_iR_i > v)} \right) dv  \right] \label{eq:Fubini} \\
&= E \left[   \int_0^{\sum_{i=1}^N C_i R_i + Q} v^{\alpha-1} dv  -  \sum_{i=1}^N \int_0^{C_i R_i} v^{\alpha-1}  dv  \right] \label{eq:DiffIntegrals} \\
&= \frac{1}{\alpha} E\left[ \left( \sum_{i=1}^N C_i R_i + Q \right)^\alpha - \sum_{i=1}^N (C_i R_i)^\alpha   \right] , \notag
\end{align}
where \eqref{eq:Fubini} is justified by Fubini's Theorem and the absolute integrability of $v^{\alpha-1} (P(R > v) - E[N] P(CR > v))$, and \eqref{eq:DiffIntegrals} is justified from the observation that 
$$v^{\alpha-1} 1_{(\sum_{i=1}^N C_i R_i + Q > v)} \qquad \text{and} \qquad v^{\alpha-1} \sum_{i=1}^N 1_{(C_iR_i > v)}$$
are each almost surely absolutely integrable as well. This completes the proof. 
\qed\end{proof}

\section{The case when $N$ dominates} \label{S.NDominates}

We now turn our attention to the distributional properties of $R^{(n)}$ and $R$ when $N$ has a heavy-tailed distribution (in particular, regularly varying) that is heavier than the potential power law effect arising from the multiplicative weights $\{C_i\}$. This case is particularly important for understanding the behavior of Google's PageRank algorithm since the $C_i$'s are smaller than one and the in-degree distribution of the Web graph is well accepted to be a power law.  We start this section by stating the corresponding lemma that describes the asymptotic behavior of $R^{(n)}$.  The main technical difficulty of extending this lemma to steady state ($R = R^{(\infty)}$) is to develop a uniform bound for $R-R^{(n)}$, which is enabled by our main technical result of this section, Proposition \ref{P.UniformBound}.  The following lemmas are proved in Section \ref{SS.NDominates_Proofs}. 


Before stating the lemmas, let us recall that a function $L: [0, \infty) \to (0, \infty)$ is slowly varying if $L(\lambda x)/L(x) \to 1$ as $x \to \infty$ for any $\lambda > 0$. We then say that the function $x^{-\alpha} L(x)$ is regularly varying with index $\alpha$.

\begin{lem} \label{L.Finite_n}
Suppose that $P(N > x) = x^{-\alpha} L(x)$ with $L(\cdot)$ slowly varying, $\alpha > 1$, and $E[Q^{\alpha+\epsilon}] < \infty$, $E[C^{\alpha+\epsilon}] < \infty$ for some $\epsilon > 0$. Let $\rho = E[N] E[C]$ and $\rho_\alpha = E[N] E[C^\alpha]$. Then, for any fixed $n \in \{1, 2, 3,\dots\}$,
\begin{equation} \label{eq:Asymp_R_n}
P(R^{(n)} > x) \sim  \frac{(E[C]E[Q])^\alpha}{(1-\rho)^\alpha} \sum_{k=0}^n \rho_\alpha^k (1-\rho^{n-k})^\alpha P(N > x)
\end{equation}
as $x \to \infty$, where $R^{(n)}$ was defined in Section \ref{SS.TreeConstruction}.
\end{lem}

\bigskip

\begin{lem} \label{L.W_n_Finite_n}
Suppose that $P(N > x) = x^{-\alpha} L(x)$ with $L(\cdot)$ slowly varying, $\alpha > 1$, and $E[Q^{\alpha+\epsilon}] < \infty$, $E[C^{\alpha+\epsilon}] < \infty$ for some $\epsilon > 0$. Let $\rho = E[N] E[C]$ and $\rho_\alpha = E[N] E[C^\alpha]$. Then, for any fixed $n \in \{1, 2, 3,\dots\}$,
$$P(W_n > x) \sim (E[C] E[Q])^\alpha  \sum_{k=0}^{n-1} \rho_\alpha^k  \rho^{(n-1-k)\alpha} P(N > x)$$
as $x \to \infty$, where $W_n$ was defined in Section \ref{SS.TreeConstruction}.
\end{lem}

From this result, provided $\rho \vee \rho_\alpha < 1$, it is to be expected that a bound of the form
$$P(W_n > x) \leq K \eta^n P(N > x)$$
might hold for all $n$ and $x \geq 1$, for some $\rho \vee \rho_\alpha < \eta < 1$. Such a bound will provide the necessary tools to ensure that $R-R^{(n)}$ is negligible for large enough $n$, allowing the exchange of limits in Lemma \ref{L.Finite_n}. Proving this result is the main technical contribution of this section; the actual proof is given in Section \ref{SS.NDominates_Proofs}. This bound may be of independent interest for computing the distributional properties of other recursions on branching trees, e.g. it is straightforward to apply our method to study the solution to 
$$R = Q + \max_{1 \leq i \leq N} C_i R_i,$$
and similar recursions. 

\bigskip

\begin{prop} \label{P.UniformBound}
Suppose $P(N > x) = x^{-\alpha} L(x)$, with $L(\cdot)$ slowly varying and $\alpha > 1$, $E[C^{\alpha+\nu}] < \infty$, $E[Q^{\alpha+\nu}] < \infty$ for some $\nu > 0$, and let $E[N] \max\{E[C^{\alpha}], E[C]\} < \eta < 1$.  Then, there exists a constant $K = K(\eta,\nu) > 0$ such that for all $n \geq 1$ and all $x \geq 1$,
\begin{equation} \label{eq:UniformBound}
P(W_n > x) \leq K  \eta^n P(N > x).
\end{equation}
\end{prop}

\bigskip

We would also like to point out that a bound of type \eqref{eq:UniformBound} resembles a classical result by Kesten (see Lemma~7 on p.~149 of \cite{Ath_McD_Ney_78}) stating that the sum of heavy-tailed (subexponential) random variables satisfies
$$P( X_1 + \dots + X_n > x) \leq K (1+\epsilon)^n P(X_1 > x),$$
uniformly for all $n$ and $x$, for any $\epsilon > 0$ (see also \cite{Den_Foss_Kor_09} for more recent work). The main difference between this result and \eqref{eq:UniformBound} is that while $n$ above refers to the number of terms in the sum, in \eqref{eq:UniformBound} it refers to the depth of the recursion. This makes the derivation of \eqref{eq:UniformBound} considerably more complicated, and perhaps implausible if it were not for the fact that we restrict our attention to regularly varying distributions, as opposed to the general subexponential class.

In view of \eqref{eq:UniformBound}, we can now prove the main theorem of this section.

\begin{thm} \label{T.Main_N}
Suppose $P(N > x) = x^{-\alpha} L(x)$, with $L(\cdot)$ slowly varying and $\alpha > 1$. Let $\rho = E[N] E[C]$ and $\rho_\alpha = E[N] E[C^\alpha]$. Assume
$\rho \vee \rho_\alpha < 1$, and $E[C^{\alpha+\epsilon}] < \infty$, $E[Q^{\alpha+\epsilon}] < \infty$ for some $\epsilon > 0$. Then, 
$$P(R > x) \sim \frac{(E[C]E[Q])^\alpha}{(1-\rho)^\alpha(1-\rho_\alpha)} P(N > x)$$
as $x \to \infty$, where $R$ was defined by \eqref{eq:R_Def}.  
\end{thm}

{\sc Remarks:} (i) A related result that also allows $Q$ and $N$ to be dependent was derived very recently in \cite{Volk_Litv_08} using transform methods and tauberian theorems under the moment conditions $E[Q]< 1$, $E[C] = (1-E[Q])/E[N]$. (ii) Note that this result implies the classical result on the busy period of an M/G/1 queue derived in \cite{Mey_Teug_80}. Specifically, the total number of customers in a busy period $B$ satisfies the recursion $B \stackrel{\mathcal{D}}{=} 1 + \sum_{i=1}^{N(S)} B_i$, where the $B_i$'s are iid copies of $B$, $N(t)$ is a Poisson process of rate $\lambda$ and $S$ is the service distribution; $\{B_i\}$, $N(t)$ and $S$ are mutually independent and $\rho = E[N(S)] < 1$. Now, the recursion for $B$ is obtained from our theorem by setting $C \equiv 1$ and $Q \equiv 1$,  implying that $P(B > x) \sim P(N(S) > x)/ (1-\rho)^{\alpha+1}$. Next, one can obtain the asymptotics for the length of the busy period $P$ by using the identity $B = N(P)$. This can be easily derived, in spite of the fact that $N(t)$ and $P$ are correlated, since $N(t)$ is highly concentrated around its mean. For recent work on the power law asymptotics of the GI/GI/1 busy period see \cite{Zwart_01}. (iii) In view of Lemma \ref{L.Finite_n}, the theorem shows that the limits $\lim_{x \to \infty} \lim_{n \to \infty} P(R^{(n)} > x)/ P(N > x)$ are interchangeable. 

\begin{proof}[Proof of Theorem \ref{T.Main_N}]
Fix $0< \delta <1$ and $n_0 \geq 1$. Choose $\rho \vee \rho_\alpha < \eta < 1$ and use Proposition~\ref{P.UniformBound} to obtain that for some constant $K_0 > 0$, 
$$P(W_n > x) \leq K_0 \eta^n P(N > x)$$
for all $n \geq 1$ and all $x \geq 1$. Let $H_\alpha^{(n)} = (E[C]E[Q])^\alpha (1-\rho)^{-\alpha} \sum_{k=0}^n \rho_\alpha^k (1-\rho^{n-k})^\alpha$ and $H_\alpha = H_\alpha^{(\infty)}$.  Then,
\begin{align}
&\left| P(R > x) -   H_\alpha P(N > x) \right| \notag \\
&\leq \left| P(R > x) -  P(R^{(n_0)} > x) \right| \label{eq:Tail} \\
&\hspace{5mm} + \left| P(R^{(n_0)} > x) -   H_\alpha^{(n_0)} P(N > x) \right| \label{eq:FiniteIterations} \\
&\hspace{5mm} + \left| H_\alpha^{(n_0)} - H_\alpha \right| P(N > x). \label{eq:FiniteError}
\end{align}
By Lemma \ref{L.Finite_n}, there exists a function $\varphi(x) \downarrow 0$ as $x \to \infty$ such that 
$$ \left| P(R^{(n_0)} > x) -   H_\alpha^{(n_0)} P(N > x) \right| \leq \varphi(x) H_\alpha P(N > x).$$
To bound \eqref{eq:Tail} let $\beta = \eta^{1/(2\alpha+2)} < 1$ and note that
\begin{align*}
\left| P(R > x) -  P(R^{(n_0)} > x) \right| &\leq  P\left(R^{(n_0)} + (R-R^{(n_0)}) > x, \, R-R^{(n_0)} \leq \delta x\right)  -  P(R^{(n_0)} > x) \\
&\hspace{5mm} + P\left(R-R^{(n_0)} > \delta x\right) \\
&\leq P(R^{(n_0)} > (1-\delta) x) - P(R^{(n_0)} > x) + P\left( \sum_{n = n_0+1}^\infty W_n > \delta x \right)  \\
&\leq P(R^{(n_0)} > (1-\delta) x) - H_\alpha^{(n_0)} P(N > (1-\delta) x) \\
&\hspace{5mm} + H_\alpha^{(n_0)} P(N > x) - P(R^{(n_0)} > x)  \\
&\hspace{5mm} + H_\alpha^{(n_0)} P(N > (1-\delta) x) - H_\alpha^{(n_0)} P(N > x) \\
&\hspace{5mm} + \sum_{n=n_0+1}^\infty  P\left( W_n > \delta x 
(1-\beta) \beta^{n-n_0-1} \right) \\
&\leq \left\{ 2\varphi((1-\delta)x) \frac{P(N > (1-\delta) x)}{P(N > x)} \right. \\
&\hspace{5mm} + \left. \left( \frac{P(N > (1-\delta)x)}{P(N>x)}  - 1 \right) \right\} H_\alpha P(N > x) \\
&\hspace{5mm}  + \sum_{n = n_0+1}^\infty K_0 \eta^n P\left( N > \delta x 
(1-\beta) \beta^{n-n_0-1}  \right) ,
\end{align*}
where in the last inequality we applied the uniform bound from Proposition \ref{P.UniformBound}. The expression in curly brackets is bounded by
$$2 \varphi((1-\delta)x) (1-\delta)^{-\alpha} \frac{L((1-\delta)x)}{L(x)} + \left( (1-\delta)^{-\alpha} \frac{L((1-\delta)x)}{L(x)} -1 \right) \to (1-\delta)^{-\alpha} - 1$$
as $x \to \infty$.  By Potter's Theorem (see Theorem 1.5.6 (ii) on p. 25 in \cite{BiGoTe1987}), there exists a constant $A = A(1) > 1$ such that
\begin{align*}
&\sum_{n = n_0+1}^\infty K_0 \eta^n P\left( N > \delta x 
(1-\beta) \beta^{n-n_0-1}  \right) \\
&\leq K_0 A \sum_{n = n_0+1}^\infty \eta^n \left( \delta 
(1-\beta) \beta^{n-n_0-1} \right)^{-\alpha-1}  P(N > x) \\
&= K_0 A (\delta(1-\beta))^{-\alpha-1}  (1-\eta^{1/2})^{-1} \eta^{n_0+1}  P(N > x) \\
&\leq K \delta^{-\alpha-1} \eta^{n_0} P(N > x) .
\end{align*}
Next, for \eqref{eq:FiniteError} simply note that
\begin{align*}
&\frac{1}{H_\alpha} \left| H_\alpha^{(n_0)} - H_\alpha \right| \\
&= (1-\rho_\alpha) \left( \sum_{k=0}^\infty \rho_\alpha^k - \sum_{k=0}^{n_0} \rho_\alpha^k (1-\rho^{n_0-k})^\alpha  \right)  \\
&= (1-\rho_\alpha) \sum_{k=0}^{n_0} \rho_\alpha^k (1 - (1-\rho^{n_0-k})^\alpha ) + (1-\rho_\alpha) \sum_{k=n_0+1}^\infty \rho_\alpha^k \\
&\leq (1-\rho_\alpha) \sum_{k=0}^{n_0} \rho_\alpha^k \alpha \rho^{n_0-k}  + \rho_\alpha^{n_0+1} \\
&\leq [\alpha (1-\rho_\alpha) (n_0+1) + \rho_\alpha] (\rho_\alpha \vee \rho)^{n_0} \\
&\leq K \eta^{n_0} .
\end{align*}
Finally, by replacing the preceding estimates in \eqref{eq:Tail} - \eqref{eq:FiniteError}, we obtain
\begin{align*}
\lim_{x \to \infty} \left| \frac{P(R > x)}{H_\alpha P(N > x)} - 1 \right| &\leq (1-\delta)^{-\alpha} - 1 + K \delta^{-\alpha-1} \eta^{n_0} .
\end{align*}
Since the right hand side can be made arbitrarily small by first letting $n_0 \to \infty$ and then $\delta \downarrow 0$, the result of the theorem follows. 
\qed\end{proof}

{\bf Engineering implications.}  Recall that for Google's PageRank algorithm the weights are given by $C_i = c/D_i < 1$, where $0 < c < 1$ is a constant related to the damping factor and the number of nodes in the Web graph, and $D_i$ corresponds to the out-degree of a page. We point out that dividing the ranks of neighboring pages by their out-degree has the purpose of decreasing the contribution of pages with highly inflated referencing. However, Theorem \ref{T.Main_N} reveals that the page rank is essentially insensitive to the parameters of the out-degree distribution, which means that PageRank basically reflects the popularity vote given by the number of references $N$. This same observation was previously made in \cite{Volk_Litv_08}. 

Furthermore, Theorem \ref{T.Goldie} clearly shows that the choice of weights $C_i$ in the ranking algorithm can determine the distribution of $R$ as well. Note that for the PageRank algorithm the weights $C_i = c/D_i < 1$ can never dominate the asymptotic behavior  of $R$ when $N$ is a power law. Therefore, Theorem \ref{T.Goldie} suggests a potential development of new ranking algorithms where the ranks will be much more sensitive to the weights.

\section{The case when $Q$ dominates} \label{S.QDominates}

This section of the paper treats the case when the heavy-tailed behavior of $R$ arises from the $\{Q_i\}$, known in the autoregressive processes literature as innovations. The results presented here are very similar to those in Section \ref{S.NDominates}, and so are their proofs. We will therefore only present the statements of the results and skip most of the proofs. We start with the equivalent of Lemmas \ref{L.Finite_n}  and \ref{L.W_n_Finite_n} in this context;  their proofs are given in Section \ref{SS.QDominate_Proofs}. 

\begin{lem} \label{L.Finite_nQ}
Suppose $P(Q > x) = x^{-\alpha} L(x)$ with $L(\cdot)$ slowly varying, $\alpha > 1$, and $E[N^{\alpha+\epsilon}] < \infty$, $E[C^{\alpha+\epsilon}] < \infty$ for some $\epsilon > 0$; let $\rho_\alpha = E[N] E[C^\alpha]$ . Then, for any fixed $n \in \{1,2,3,\dots\}$,
$$P(R^{(n)} > x) \sim \sum_{k=0}^n \rho_\alpha^k \, P(Q > x)$$
as $x \to \infty$, where $R^{(n)}$ was defined in Section \ref{SS.TreeConstruction}. 
\end{lem}

As for the case when $N$ dominates the asymptotic behavior of $R$, we can here expect that
$$P(R > x) \sim (1-\rho_\alpha)^{-1} P(Q > x),$$
and the technical difficulty is justifying the exchange of limits. The same techniques used in Section~\ref{S.NDominates} can be used in this case as well.  Therefore, we give a sketch of the arguments in Section~\ref{SS.QDominate_Proofs} but omit the proof. The following is the equivalent of Lemma \ref{L.W_n_Finite_n}. 

\bigskip

\begin{lem} \label{L.W_n_Finite_nQ}
Suppose $P(Q > x) = x^{-\alpha} L(x)$ with $L(\cdot)$ slowly varying, $\alpha > 1$, and $E[N^{\alpha+\epsilon}] < \infty$, $E[C^{\alpha+\epsilon}] < \infty$ for some $\epsilon > 0$; let $\rho_\alpha = E[N] E[C^\alpha]$. Then, for any fixed $n \in \{1, 2, 3,\dots\}$,
$$P(W_n > x) \sim \rho_\alpha^n P(Q > x)$$
as $x \to \infty$, where $W_n$ was defined in Section \ref{SS.TreeConstruction}.
\end{lem}

\bigskip

The corresponding version of Proposition \ref{P.UniformBound} is given below.

\begin{prop} \label{P.UniformBoundQ}
Suppose $P(Q > x) = x^{-\alpha} L(x)$, with $L(\cdot)$ slowly varying and $\alpha > 1$, $E[C^{\alpha + \nu}] < \infty$, $E[N^{\alpha+\nu}] < \infty$ for some $\nu > 0$, and let $E[N] \max\{ E[C^{\alpha}] , E[C] \} < \eta < 1$. Then, there exists a constant $K = K(\eta,\nu) > 0$ such that for all $n \geq 1$ and all $x \geq 1$,
$$P(W_n > x) \leq K  \eta^n P(Q > x).$$
\end{prop}

A sketch of the proof can be found in Section \ref{SS.QDominate_Proofs}.

\bigskip

And finally, the main theorem of this section. The proof again greatly resembles that of Theorem~ \ref{T.Main_N} and is therefore omitted. 

\begin{thm} \label{T.MainQ}
Suppose $P(Q > x) = x^{-\alpha} L(x)$, with $L(\cdot)$ slowly varying and $\alpha > 1$. Let $\rho = E[N] E[C]$ and $\rho_\alpha = E[N] E[C^\alpha]$. Assume $\rho \vee \rho_\alpha < 1$, and $E[C^{\alpha+\epsilon}] < \infty$, $E[N^{\alpha+\epsilon}] < \infty$ for some $\epsilon > 0$. Then, 
$$P(R > x) \sim (1-\rho_\alpha)^{-1} P(Q > x)$$
as $x \to \infty$, where $R$ was defined in \eqref{eq:R_Def}.  
\end{thm}

Compare this result with Lemma A.3 in \cite{Mik_Sam_00}, where the autoregressive process of order one with regularly varying innovations is shown to be tail-equivalent to $Q$. In particular, if we set $N \equiv 1$ in Theorem \ref{T.MainQ} and let $A_k = \prod_{i=1}^{k-1} C_i$, our result reduces to 
$$P\left( \sum_{k=0}^\infty A_k Q_k > x \right) \sim \sum_{k=0}^\infty E[A_k^\alpha] P(Q > x),$$
which is in line with the commonly accepted intuition about heavy-tailed large deviations where large sums are due to one large summand $Q_k$.

\section{Proofs}\label{S.Proofs}

This section contains the proofs to most of the technical results presented in the paper, together with some auxiliary lemmas that are needed along the way.  The section is divided into four subsections, each corresponding to the content of Sections \ref{S.Moments}, \ref{S.C_dominates}, \ref{S.NDominates}, and \ref{S.QDominates}, respectively.

\subsection{Moments of $W_n$} \label{SS.Moments_Proofs}

Here we give the proof of the moment bound for the $\beta$-moment, $\beta > 1$, of the sum of the weights, $W_n$ of the $n$th generation. As an intermediate step, we present a lemma for the integer moments of $W_n$, but first we give the proof of Lemma \ref{L.Alpha_Moments}.

\begin{proof}[Proof of Lemma \ref{L.Alpha_Moments}]
Let $p = \lceil \beta \rceil \in \{2,3,\dots\}$ and $\gamma = \beta/p \in (0, 1]$. Define  $A_p(k) = \{ (j_1, \dots, j_k) \in \mathbb{Z}^k: j_1 + \dots + j_k = p, 0 \leq j_i < p\}$. Then,
\begin{align}
\left( \sum_{i=1}^k y_i \right)^\beta &= \left( \sum_{i=1}^k y_i \right)^{p \gamma} \notag \\
&= \left( \sum_{i=1}^k y_i^p + \sum_{(j_1,\dots,j_k) \in A_p(k)} \binom{p}{j_1,\dots,j_k} y_1^{j_1} \cdots y_k^{j_k} \right)^\gamma \notag  \\
&\leq \sum_{i=1}^k y_i^{p\gamma} + \left( \sum_{(j_1,\dots,j_k) \in A_p(k)} \binom{p}{j_1,\dots,j_k} y_1^{j_1} \cdots y_k^{j_k} \right)^\gamma, \notag
\end{align}
where for the last step we used the well known inequality $\left( \sum_{i=1}^k x_i \right)^\gamma \leq \sum_{i=1}^k x_i^\gamma$ for $0 < \gamma \leq 1$ and $x_i \geq 0$ (see the proof of Lemma \ref{L.MomentSmaller_1}). We now use Jensen's inequality to obtain 
\begin{align*}
E\left[ \left( \sum_{i=1}^k Y_i \right)^\beta - \sum_{i=1}^k Y_i^{\beta} \right] &\leq E\left[  \left( \sum_{(j_1,\dots,j_k) \in A_p(k)} \binom{p}{j_1,\dots,j_k} Y_1^{j_1} \cdots Y_k^{j_k} \right)^\gamma \right] \\
&\leq   \left( E\left[ \sum_{(j_1,\dots,j_k) \in A_p(k)} \binom{p}{j_1,\dots,j_k} Y_1^{j_1} \cdots Y_k^{j_k} \right] \right)^\gamma \\
&= \left(  \sum_{(j_1,\dots,j_k) \in A_p(k)} \binom{p}{j_1,\dots,j_k} E\left[ Y_1^{j_1} \cdots Y_k^{j_k} \right] \right)^\gamma.
\end{align*}
Since the $\{Y_i\}$ are iid, we have
$$E\left[ Y_1^{j_1} \cdots Y_k^{j_k} \right] = || Y||_{j_1}^{j_1} \cdots ||Y||_{j_k}^{j_k},$$
where $|| Y||_\kappa = E[|Y|^\kappa]^{1/\kappa}$ for $\kappa \geq 1$ and $|| Y ||_0 \triangleq 1$. Since $|| Y||_\kappa$ is increasing for $\kappa \geq 1$ it follows that $|| Y ||_{j_i}^{j_i} \leq || Y ||_{p-1}^{j_1}$. It follows that
$$|| Y||_{j_1}^{j_1} \cdots ||Y||_{j_k}^{j_k} \leq || Y ||_{p-1}^p,$$
which in turn implies that
\begin{align*}
E\left[ \left( \sum_{i=1}^k Y_i \right)^\beta - \sum_{i=1}^k Y_i^{\beta} \right] &\leq \left( \sum_{(j_1,\dots,j_k) \in A_p(k)} \binom{p}{j_1,\dots,j_k} || Y ||_{p-1}^p \right)^\gamma \\
&= || Y||_{p-1}^\beta (k^p - k)^\gamma \\
&\leq || Y ||_{p-1}^\beta k^\beta.
\end{align*}
\qed\end{proof}

\bigskip

\begin{lem} \label{L.IntegerMoment}
Suppose $E[Q^p]< \infty$, $E[N^p] < \infty$, and $E[N] \max\{ E[C^p], E[C] \} < 1$ for some $p \in \{2,3,\dots\}$. Then, there exists a constant $K_p > 0$ such that
$$E[ W_n^p ] \leq K_p \left( E[N] \max\{E[C], E[C^p]\} \right)^n $$
for all $n \geq 0$.  \vspace{-10pt}
\end{lem}

\begin{proof}
Let $Y = C W_{n-1}$, where $C$ is independent of $W_{n-1}$ and let $\{Y_i\}$ be independent copies of $Y$. We will give an induction proof in $p$. For $p = 2$ we have
\begin{align*}
E[W_n^2] &= E\left[ \left( \sum_{i=1}^N Y_i \right)^2  \right] \\
&= E[N] E[Y^2] + E[N(N-1)] (E[Y])^2 \\
&= E[N] E[C^2] E[W_{n-1}^2] + E[N(N-1)] (E[C] E[W_{n-1}])^2 .
\end{align*}
Using the preceding recursion, letting $\rho = E[N] E[C]$, $\rho_2 = E[N] E[C^2]$, and noting that,
$$E[W_{n-1}] = \rho^{n-1} E[Q],$$
we obtain
\begin{equation} \label{eq:2_recur}
E[W_n^2] = \rho_2 E[W_{n-1}^2] + K \rho^{2(n-1)},
\end{equation}
where $K = E[N(N-1)] (E[C] E[Q])^2$. Now, iterating \eqref{eq:2_recur} gives
\begin{align*}
E[W_n^2] &= \rho_2 \left( \rho_2 E[W_{n-2}^2] +  K \rho^{2(n-2)} \right) +  K \rho^{2(n-1)} \\
&= \rho_2^{n-1} \left( \rho_2 E[W_{0}^2] + K  \right) + K \sum_{i=0}^{n-2} \rho_2^i \, \rho^{2(n-1-i)} \\
&= \rho_2^n E[Q^2] + K  \sum_{i=0}^{n-1} \rho_2^i \, \rho^{2(n-1-i)}  \\
&\leq (\rho_2 \vee \rho)^n E[Q^2] + K (\rho_2 \vee \rho)^n \sum_{i=0}^{n-1} (\rho_2 \vee \rho)^{n-2 - i  }  \\
&\leq \left( E[Q^2] + \frac{K}{\rho_2 \vee \rho} \sum_{j=0}^{\infty} (\rho_2 \vee \rho)^{j}  \right) (\rho_2 \vee \rho)^n \\
&= K_2 (\rho_2 \vee \rho)^n .
\end{align*}
Next, for any $p \in \{2, 3, \dots\}$ let $\rho_p = E[N] E[C^p]$. Suppose now that there exists a constant $K_{p-1} > 0$ such that
\begin{equation} \label{eq:Induction_p}
E[W_n^{p-1}] \leq K_{p-1} \left( \rho_{p-1} \vee \rho \right)^n
\end{equation}
for all $n \geq 0$. By Lemma \ref{L.Alpha_Moments} we have
\begin{align*}
E[W_n^p] &= \sum_{k=1}^\infty E\left[ \left( \sum_{i=1}^k Y_i \right)^p \right] P(N = k) \\
&\leq  \sum_{k=1}^\infty \left( k E\left[ Y^p \right] + k^p (E[ Y^{p-1}])^{p/(p-1)} \right) P(N = k) \\
&= E[N] E[ C^p] E[W_{n-1}^p] + E[N^p] (E[C^{p-1}])^{p/(p-1)} (E[W_{n-1}^{p-1}])^{p/(p-1)} \\
&\leq \rho_p E[W_{n-1}^p] + E[N^p] (E[C^{p-1}])^{p/(p-1)}  (K_{p-1})^{p/(p-1)} (\rho_{p-1} \vee \rho)^{(n-1)p/(p-1)},
\end{align*}
where the last inequality corresponds to the induction hypothesis. We then obtain the recursion
\begin{equation} \label{eq:p_recur}
E[W_n^p] \leq \rho_p E[W_{n-1}^p] + K (\rho_{p-1} \vee \rho)^{\frac{(n-1)p}{p-1}},
\end{equation}
where $K = E[N^p] (E[C^{p-1}])^{p/(p-1)}  (K_{p-1})^{p/(p-1)}$. Iterating \eqref{eq:p_recur} as for the case $p=2$ gives
\begin{align*}
E[W_n^p] &\leq \rho_p^n E[Q^p] + K  \sum_{i=0}^{n-1} \rho_p^i \, (\rho_{p-1} \vee \rho)^{\frac{(n-1-i)p}{p-1}} \\
&\leq (\rho_p \vee \rho)^n E[Q^p] + K \sum_{i=0}^{n-1} (\rho_p \vee \rho)^{\frac{(n-1)p -i}{p-1} } \\
&= (\rho_p \vee \rho)^n E[Q^p] + K (\rho_p \vee \rho)^n   \sum_{i=0}^{n-1} (\rho_p \vee \rho)^{\frac{ n- i - p }{p-1} }   \\
&\leq  \left( E[Q^p] + K (\rho_p \vee \rho)^{-1}   \sum_{j=0}^{\infty} (\rho_p \vee \rho)^{\frac{j}{p-1}} \right) (\rho_p \vee \rho)^n \\
&= K_p (\rho_p \vee \rho)^n.
\end{align*}
\qed\end{proof}

\bigskip

The proof for the general $\beta$-moment, $\beta > 1$, is given below. 

\begin{proof}[Proof of Lemma \ref{L.GeneralMoment}]
Set $p = \lceil \beta \rceil \geq \beta > 1$. Since the result when $p = \beta$ follows from Lemma \ref{L.IntegerMoment}, we assume that $p > \beta$. Let $Y = C W_{n-1}$, where $C$ is independent of $W_{n-1}$ and $\{Y_i\}$ are independent copies of $Y$. Also, recall that $\rho = E[N] E[C]$ and $\rho_\beta = E[N] E[C^\beta]$. Then, by Lemma \ref{L.Alpha_Moments}, 
\begin{align*}
E[W_n^\beta] &= E\left[  \left( \sum_{i=1}^N Y_i \right)^\beta \right] \\
&=  \sum_{k=1}^\infty E\left[   \left( \sum_{i=1}^k Y_i \right)^{\beta}   \right]  P(N = k) \\
&=  \sum_{k=1}^\infty \left( E\left[   \left( \sum_{i=1}^k Y_i \right)^{\beta} - \sum_{i=1}^k Y_i^\beta  \right] + E\left[ \sum_{i=1}^k Y_i^\beta \right] \right) P(N = k) \\
&\leq \sum_{k=1}^\infty \left( k^\beta E[Y^{p-1}]^{\beta/(p-1)} + k E\left[ Y^\beta \right]   \right) P(N = k) \\
&= E[N^\beta]  (E[C^{p-1}])^{\beta/(p-1)} (E[W_{n-1}^{p-1}])^{\beta/(p-1)} + \rho_\beta E[ W_{n-1}^\beta] .
\end{align*}
Then, by Lemma \ref{L.IntegerMoment}, 
\begin{align*}
E[W_n^\beta] &\leq  \rho_\beta E[ W_{n-1}^\beta] + E[N^\beta]  (E[C^{p-1}])^{\beta/(p-1)} (K_{p-1} (\rho_{p-1} \vee \rho)^{n-1})^{\beta/(p-1)} \\
&= \rho_\beta E[ W_{n-1}^\beta] + K (\rho_{p-1} \vee \rho)^{(n-1)\gamma},
\end{align*}
where $\gamma = \beta/(p-1) > 1$. Finally, iterating the preceding bound $n-1$ times gives
\begin{align*}
E[W_n^\beta] &\leq \rho_\beta^n E[W_0^\beta] + K \sum_{i=0}^{n-1} \rho_\beta^i (\rho \vee \rho_{p-1})^{\gamma(n-1-i)} \\
&\leq E[W_0^\beta]  (\rho \vee \rho_\beta)^n + K  \sum_{i=0}^{n-1} (\rho \vee \rho_\beta)^{\gamma(n-1-i) + i} \\
&= E[Q^\beta] (\rho \vee \rho_\beta)^n + K (\rho \vee \rho_\beta)^{n-1} \sum_{i=0}^{n-1} (\rho \vee \rho_\beta)^{(\gamma-1) i} \\
&\leq K_\beta (\rho \vee \rho_\beta)^n .
\end{align*}
This completes the proof. 
\qed\end{proof}

\subsection{The case when the $C$'s dominate: Implicit renewal theory} \label{SS.CDominates_Proofs}

In this section we state a lemma that is used in the proof of Theorem~\ref{T.Goldie} and we give the proofs to Theorem~\ref{T.Goldie} and Lemma~\ref{L.Max_Approx}.

\begin{lem} \label{L.Derivative}
Let $\alpha, \beta > 0$ and $H \geq 0$. Suppose $\int_0^t v^{\alpha+\beta-1} P(R > v) dv \sim H t^{\beta}/\beta$ as $t \to \infty$. Then,
$$P(R > t) \sim H t^{-\alpha}, \qquad t \to \infty.$$
\end{lem}

\begin{proof}
This lemma is a special case of the Monotone Density Theorem, see Theorem 1.7.5 (also Exercise~1.11.14) in \cite{BiGoTe1987}. However, for completeness, we give a direct proof here, similar to the one of Lemma 9.3 in \cite{Goldie_91}. By assumption, for any $b > 1$, $\epsilon \in (0,1)$, and $t$ sufficiently large,
\begin{align*}
P(R > t) t^{\alpha +\beta} \cdot \frac{b^{\alpha+\beta}-1}{\alpha+\beta} &\geq \int_{t}^{b t} v^{\alpha+\beta-1} P(R > v) \, dv \geq \frac{(H-\epsilon)}{\beta} (b t)^\beta - \frac{(H+\epsilon)}{\beta} t^\beta  \\
&\geq \frac{t^\beta}{\beta} \left( H (b^\beta-1) -\epsilon (1 + b^\beta)  \right).
\end{align*}
Since $\epsilon$ was arbitrary, we can take the limit as $\epsilon \to 0$ and obtain
$$\liminf_{t \to \infty} P(R > t) t^{\alpha} \geq \frac{H (\alpha+\beta) (b^\beta-1)}{\beta(b^{\alpha+\beta} - 1)} \to H, \qquad b \downarrow 1.$$
Similarly, one can prove that $\limsup_{t \to \infty} P(R > t) t^\alpha \leq H$ starting from $\int_{bt}^t v^{\alpha+\beta -1} P(R > v) \, dv$ with $0 < b < 1$. 
\qed\end{proof}

\begin{proof}[Proof of Theorem \ref{T.Goldie}]
For any $k \in \mathbb{N}$ define $\Pi_k = \prod_{i=1}^k C_i$ and $V_k = \sum_{i=1}^k \log C_i$, with $\Pi_0 = 1$ and $V_0 = 0$, where the $C_i$'s are independent copies of $C$. Then, for any $t \in \mathbb{R}$,
\begin{align*}
P(R > e^t) &= \sum_{k=1}^n \left( m^{k-1} P(\Pi_{k-1} R > e^t) - m^k P(\Pi_k R > e^t) \right) + m^n P( \Pi_n R > e^t) \\
&= \sum_{k=1}^n  \left( m^{k-1} P(e^{V_{k-1}} R > e^t) - m^k P( e^{V_{k-1}} C_k R > e^t) \right) + m^n P( e^{V_n} R > e^t) \\
&= \sum_{k=0}^{n-1} m^{k} \int_{-\infty}^\infty \left( P( R > e^{t-v}) - mP( C R > e^{t-v}) \right) P(V_{k} \in dv) + m^n P( e^{V_n} R > e^t).
\end{align*}
Next, define
$$\nu_n(dt) = e^{\alpha t} \sum_{k=0}^n m^{k} P(V_k \in dt), \qquad g(t) = e^{\alpha t} (P(R > e^t) - m P(CR > e^t)),$$
$$r(t) = e^{\alpha t} P(R > e^t)  \qquad \text{and} \qquad \delta_n(t) = m^n P( e^{V_n} R > e^t).$$
Then, for any $t \in \mathbb{R}$ and $n \in \mathbb{N}$, 
$$r(t) = (g*\nu_{n-1})(t) + \delta_n(t).$$
Next, for any $\beta > 0$, define the smoothing operator
$$\breve{f}(t) = \int_{-\infty}^t e^{-\beta(t-u)} f(u) \, du$$
and note that
\begin{align}
\breve{r}(t) &= \int_{-\infty}^t e^{-\beta(t-u)} (g*\nu_{n-1})(u) \, du + \breve{\delta}_n (t) \notag \\
&= \int_{-\infty}^t e^{-\beta(t-u)} \int_{-\infty}^\infty g(u-v) \nu_{n-1}(dv) \, du + \breve{\delta}_n (t) \notag \\
&= \int_{-\infty}^\infty \int_{-\infty}^t e^{-\beta(t-u)} g(u-v) \, du \, \nu_{n-1}(dv) + \breve{\delta}_n (t) \notag \\
&= \int_{-\infty}^\infty \breve{g}(t-v) \, \nu_{n-1}(dv) + \breve{\delta}_n (t) \notag \\
&= (\breve{g}* \nu_{n-1})(t) + \breve{\delta}_n(t) . \label{eq:SmoothOperator}
\end{align}

Next, we will show that one can pass $n \to \infty$ in the preceding identity. To this end, let $\eta(du) = e^{\alpha u} m P(\log C \in du)$, and note that this measure places no mass at $-\infty$. Also, by assumption, $\eta(\cdot)$ is a nonarithmetic measure on $\mathbb{R}$. Moreover,
$$\int_{-\infty}^\infty \eta(du) = m E[ e^{\alpha \log C}] = m E[ C^\alpha] = 1$$
and
$$\int_{-\infty}^\infty u\, \eta(du) = m E[ e^{\alpha \log C} \log C] = m E[C^\alpha \log C] = m\mu$$
imply that $\eta(\cdot)$ is a probability measure with mean $0 < m\mu < \infty$. Furthermore,
$$\nu(dt) = \sum_{k=0}^\infty m^k e^{\alpha t} P(V_k \in dt)$$
is its renewal measure since $\nu(dt) = \sum_{n=0}^\infty \eta^{*n}(dt)$. Since $m\mu > 0$, then $(|f|*\nu)(t) < \infty$ for all $t$ whenever $f$ is directly Riemann integrable. From \eqref{eq:Goldie_condition} we know that $g \in L^1$, so by Lemma 9.2 from \cite{Goldie_91}, $\breve{g}$ is directly Riemann integrable, resulting in $(|\breve{g}|*\nu)(t) < \infty$ for all $t$. Thus $(|\breve{g}|*\nu)(t) = E\left[ \sum_{k=0}^\infty m^k e^{\alpha V_k} | \breve{g}(t - V_k) | \right] < \infty$. By Fubini's Theorem, 
$E\left[ \sum_{k=0}^\infty m^k e^{\alpha V_k}  \breve{g}(t - V_k)  \right] $ exists and 
$$(\breve{g}*\nu)(t) = E\left[ \sum_{k=0}^\infty m^k e^{\alpha V_k}  \breve{g}(t - V_k)  \right]  = \sum_{k=0}^\infty E\left[  m^k e^{\alpha V_k}  \breve{g}(t - V_k) \right] = \lim_{n\to \infty}  (\breve{g}*\nu_n)(t).$$
Now, by assumption, we can choose $\beta$ in the definition of the smoothing operator such that $0 < \beta < \alpha$ and $m E[C^\beta] < 1$. We show below that for such $\beta$ we have $\breve{\delta}_n(t) \to 0$ as $n \to \infty$ for all fixed $t$, since
\begin{align*}
\breve{\delta}_n(t) &= \int_{-\infty}^t e^{-\beta(t-u)} m^n P(e^{\beta V_n} R^\beta > e^{\beta u}) \, du \\
&= \frac{e^{-\beta t} m^n}{\beta} \int_0^{e^{\beta t}} P( e^{\beta V_n} R^\beta > v) \, dv \ \\
&\leq \frac{e^{-\beta t}}{\beta} E[R^\beta] (m E[C^\beta])^n  \to 0 
\end{align*}
as $n \to \infty$. Hence, the preceding arguments allow us to pass $n \to \infty$ in \eqref{eq:SmoothOperator}, and obtain
$$\breve{r}(t) = (\breve{g}*\nu)(t).$$

Now, by the key renewal theorem for two-sided random walks in \cite{Ath_McD_Ney_78},
$$e^{-\beta t} \int_{0}^{e^t} v^{\alpha+\beta-1} P(R > v) \, dv = \breve{r}(t) \to \frac{1}{m\mu} \int_{-\infty}^\infty \breve{g}(u) \, du \triangleq \frac{H}{\beta}, \qquad t \to \infty.$$
Clearly, $H \geq 0$ since the left-hand side of the preceding equation is positive, and thus, by Lemma \ref{L.Derivative} 
$$P(R > t) \sim H t^{-\alpha}, \qquad t \to \infty.$$
Finally, 
\begin{align*}
H &=  \frac{\beta}{m\mu} \int_{-\infty}^\infty \int_{-\infty}^u e^{-\beta (u-t)} g(t) \, dt \, du \\
&= \frac{1}{m\mu} \int_{-\infty}^\infty  g(t) \, dt \\
&= \frac{1}{m\mu} \int_{0}^\infty v^{\alpha -1} (P(R > v) - mP(C R > v)) \, dv.
\end{align*}
\qed\end{proof}

\bigskip

We end this section with the proof of Lemma \ref{L.Max_Approx}.

\begin{proof}[Proof of Lemma \ref{L.Max_Approx}]
That the integral is positive follows from the union bound. That 
$$\int_0^\infty \left( E[N] P(CR > t) - P\left( \max_{1\leq i\leq N} C_i R_i > t \right) \right) t^{\alpha-1} dt = \frac{1}{\alpha} E\left[ \sum_{i=1}^N (C_iR_i)^\alpha - \left( \max_{1\leq i\leq N} C_iR_i \right)^\alpha \right]$$
follows from similar arguments to those used to derive the alternative expression for $H$ in the proof of Theorem \ref{T.GoldieApplication}. The rest of the proof shows that the integral is finite.

Clearly
\begin{align*}
\int_{0}^{1} \left( E[N] P(C R > t) - P\left( \max_{1\leq i \leq N} C_i R_i > t \right) \right)  t^{\alpha -1} \, dt &\leq E[N] \int_0^{1} t^{\alpha-1} \, dt < \infty.
\end{align*}
Hence, it remains to prove that the remaining part of the integral $\left( \int_{1}^\infty \cdots dt \right)$ is finite. To do this, we start by letting $Y = CR$ and $F(y) = P(Y \leq y)$. Then
\begin{align*}
E[N] P(CR > t) - P\left( \max_{1\leq i \leq N} C_i R_i > t \right)&= \sum_{k=1}^\infty \left( F(t)^k -1 +  k \overline{F}(t)  \right) P(N = k)  \\
&= E\left[ (1-\overline{F}(t))^N  - 1 + N \overline{F}(t) \right] .
\end{align*}
Use the inequality $1 - x \leq e^{-x}$ for $x \geq 0$ to obtain
$$E\left[ (1-\overline{F}(t))^N  - 1 + N \overline{F}(t) \right] \leq E\left[ e^{-\overline{F}(t) N}  - 1 + N \overline{F}(t) \right].$$
Choose $0 < \delta < \alpha\epsilon/(1+\epsilon)$ (recall that $0 < \epsilon < 1$) and let $\beta = \alpha-\delta$. By Markov's inequality and Lemma \ref{L.Stability} 
$$\overline{F}(t) \leq t^{-\beta} E[ Y^\beta] = t^{-\beta} E[R^\beta] E[C^\beta] \triangleq c t^{-\beta} < \infty$$
for any $t > 0$. Note that the function $h(x) = e^{-x} - 1 + x$ is increasing on $[0,\infty)$, so $h(N \overline{F}(t)) \leq h(cN t^{-\beta})$. Thus, by Fubini's Theorem (the integrand is nonnegative),
$$\int_{1}^\infty \left( E[N] P(C R > t) - P\left( \max_{1\leq i \leq N} C_i R_i > t \right) \right)  t^{\alpha -1} \, dt  \leq E\left[ \int_{1}^\infty \left( e^{-cN t^{-\beta}}  - 1 + cN t^{-\beta} \right)  t^{\alpha -1} \, dt   \right].$$ 
Using the change of variables $u = c N t^{-\beta}$ gives
\begin{align*}
\int_{1}^\infty \left( e^{-cN t^{-\beta}}  - 1 + c N t^{-\beta} \right)  t^{\alpha -1} \, dt &= \frac{(cN)^{\alpha/\beta}}{\beta} \int_0^{cN} \left( e^{-u} - 1 + u \right) u^{-\alpha/\beta  -1} \, du \\
&\leq  \frac{(cN)^{\alpha/\beta}}{\beta} \int_0^\infty \left( e^{-u} - 1 + u \right) u^{-\alpha/\beta  -1} \, du.
\end{align*}
Our choice of $\beta = \alpha-\delta$ guarantees that $1 < \alpha/\beta < 1+\epsilon$, so $E[(cN)^{\alpha/\beta}] < \infty$. It only remains to show that the last (non-random) integral is finite. To see this note that $e^{-x} -1 + x \leq x^2/2$ and $e^{-x} - 1 \leq 0$ for any $x \geq 0$, so
\begin{align*}
\int_0^\infty \left( e^{-u} - 1 + u \right) u^{-\alpha/\beta  -1} \, du &\leq \frac{1}{2} \int_0^1  u^{1-\alpha/\beta} \, du + \int_1^\infty  u^{-\alpha/\beta } \, du \\
&= \frac{1}{2(2-\alpha/\beta)} + \frac{1}{\alpha/\beta-1} < \infty.
\end{align*}
This completes the proof. 
\qed\end{proof}

\subsection{The case when $N$ dominates} \label{SS.NDominates_Proofs}

This section contains the proofs of Lemma \ref{L.Finite_n} and Proposition \ref{P.UniformBound}; the proof of Lemma \ref{L.W_n_Finite_n} is omitted since it is basically the same as that of Lemma \ref{L.Finite_n}.  We also present in Lemma \ref{L.TruncBound} a result for sums of iid truncated random variables that may be of independent interest in the context of heavy-tailed asymptotics, since it provides bounds that do not depend on the distribution of the summands. Most of the work involved in the proof of Proposition \ref{P.UniformBound} goes into obtaining a bound for one iteration of the recursion satisfied by $W_n$, and for the convenience of the reader it is presented separately in Lemma \ref{L.Bound1Iter}.

\begin{proof}[Proof of Lemma \ref{L.Finite_n}]
We proceed by induction in $n$. For $n = 1$ fix $\alpha/(\alpha+\epsilon) < \delta < 1$ and note that
\begin{align*}
P(R^{(1)} > x) &= P\left( \sum_{i=1}^{N^{(0)}} C_i^{(1)} R_{i}^{(0)} + Q^{(0)} > x \right) \\
&= P\left( \sum_{i=1}^N C_i Q_i > x - Q, \, Q \leq x^\delta \right) + P\left( Q >  x^\delta \right) \\
&\sim P\left( \sum_{i=1}^N C_i Q_i > x \right) + O\left( x^{-\delta(\alpha+\epsilon)} \right) \\
&\sim P( N > x/ E[CQ] ) + o(P(N > x))  \\
&\sim (E[C] E[Q])^\alpha P(N > x),
\end{align*}
where $N, \{C_i\},$ and $Q$ are independent and the fourth step is justified by Lemma 3.7(2) from \cite{Jess_Miko_06}. Now suppose that we have
$$P(R^{(n)} > x) \sim \frac{(E[C]E[Q])^\alpha}{(1-\rho)^\alpha} \sum_{k=0}^n \rho_\alpha^k (1-\rho^{n-k})^\alpha P(N > x).$$
Note that since $E[C^{\alpha + \epsilon}] < \infty$, then by Lemma 4.2 from \cite{Jess_Miko_06}, for $C$ independent of $R^{(n)}$, 
$$P(C R^{(n)} > x) \sim E[C^\alpha] P(R^{(n)} > x).$$
Let $c^{-1} = E[C^\alpha] (E[C]E[Q])^\alpha (1-\rho)^{-\alpha} \sum_{k=0}^n \rho_\alpha^k (1-\rho^{n-k})^\alpha$, then 
$$P(N > x) \sim c P(C R^{(n)} > x),$$
and by Lemma 3.7(5) from \cite{Jess_Miko_06} we have
\begin{align*}
P(R^{(n+1)} > x) &= P\left( \sum_{i=1}^{N^{(0)}} C_i^{(1)} R_{i}^{(n)} + Q^{(0)} > x \right) \\
&\sim P\left( \sum_{i=1}^{N^{(0)}} C_i^{(1)} R_{i}^{(n)} > x \right) \\
&\sim (E[N] + c (E[C R^{(n)}])^\alpha) P( C R^{(n)} > x) \\
&\sim (E[N] + c (E[C R^{(n)}])^\alpha) c^{-1} P( N > x) .
\end{align*}
Next, observing that $E[R^{(n)}] = \sum_{i=0}^n E[W_i] = E[Q] \sum_{i=0}^n \rho^i = E[Q] (1-\rho^{n+1})/(1-\rho)$, we obtain
\begin{align*}
(E[N] + c (E[C R^{(n)}])^\alpha) c^{-1} &= \left(\rho_\alpha+ \frac{ E[R^{(n)}]^\alpha (1-\rho)^\alpha}{ E[Q]^\alpha \sum_{k=0}^n \rho_\alpha^k (1-\rho^{n-k})^\alpha} \right) \frac{(E[C]E[Q])^\alpha}{(1-\rho)^\alpha} \sum_{k=0}^n \rho_\alpha^k (1-\rho^{n-k})^\alpha \\
&= \left(\rho_\alpha \sum_{k=0}^n \rho_\alpha^k (1-\rho^{n-k})^\alpha + (1-\rho^{n+1})^\alpha  \right) \frac{(E[C]E[Q])^\alpha}{(1-\rho)^\alpha} \\
&= \frac{(E[C]E[Q])^\alpha}{(1-\rho)^\alpha}  \sum_{k=0}^{n+1} \rho_\alpha^k (1-\rho^{n+1-k})^\alpha.
\end{align*}
This completes the proof. 
\qed\end{proof}

Lemma \ref{L.TruncBound} below is based on traditional heavy-tailed techniques based on Chernoff's inequality for truncated random variables, such as those used in \cite{Nag82} and \cite{Bor00}, to name some references. The reason why we cannot simply use existing results is our need to guarantee that the bounds do not depend on the distribution of the summands, which will be key when we apply them to $W_n$. Hence special care goes into accounting for the constants explicitly. The corollary that we obtain from this lemma will be used in the proof of Lemma~ \ref{L.Bound1Iter}. 

\begin{lem} \label{L.TruncBound}
Suppose that $Y_1, Y_2, \dots$ are nonnegative iid random variables with the same distribution as $Y$, where $E[Y^\beta] < \infty$ for some $\beta > 0$.  Fix $0 < \epsilon < 1$. Then,
\begin{enumerate}
\item for $0 < \beta < 1$, $1 \leq k \leq x^\beta/ E[Y^\beta]$, and $x \geq e^{(Ke)^{1/(1-\beta)}}$, 
$$P\left( \sum_{i=1}^k Y_i  > x, \, \max_{1\leq i \leq k} Y_i \leq x/\log x \right) \leq  e^{-(1-\beta)(\log x)(\log\log x) \left( 1 - \frac{\log (eK)}{(1-\beta)\log\log x} \right) },$$
\item for $\beta > 1$, $1 \leq k \leq (1-\epsilon) x/(E[Y] \vee E[Y^\beta])$, and $x \geq e \vee (Ke/\epsilon)^{2/(\beta-1)}$, 
$$P\left( \sum_{i=1}^k Y_i  > x, \, \max_{1\leq i \leq k} Y_i \leq  x/\log x \right) \leq  e^{-\epsilon (\beta-1) (\log x)^2 \left( 1 - \frac{\log\log x}{\log x} - \frac{\log(Ke/\epsilon)}{(\beta-1)\log x} \right) + e^5 (\beta-1)^2},$$
\end{enumerate}
where $K = K(\beta) > 1$ is a constant that does not depend on $\epsilon$, $k$ or the distribution of $Y$.  
\end{lem}

\begin{proof}
Let $F(t) = P(Y \leq t)$, set $y = x/\log x$ and note that
\begin{align*}
P\left( \sum_{i=1}^k Y_i > x, \max_{1\leq i \leq k} Y_i \leq y \right) =  P\left( \sum_{i=1}^k Y_i^{(y)} > x \right) F(y)^k ,
\end{align*}
where $P( Y^{(y)} \leq t) = F(t \wedge y)/F(y)$. Fix $\theta \geq 1/y$ and use the standard Chernoff's bound method for truncated heavy tailed sums (see, e.g. \cite{Nag82,Bor00}) to obtain
\begin{align*}
P\left( \sum_{i=1}^k Y_i^{(y)} > x \right) &\leq e^{-\theta x} E\left[ e^{\theta Y_1^{(y)}} \right]^k = e^{-\theta x} E \left[ e^{\theta Y} 1_{(Y \leq y)} \right]^k F(y)^{-k}.
\end{align*}
From where it follows that
$$P\left( \sum_{i=1}^k Y_i^{(y)} > x \right) F(y)^k \leq e^{-\theta x} E \left[ e^{\theta Y} 1_{(Y \leq y)} \right]^k.$$
To analyze the preceding truncated exponential moment suppose first that $\beta > 1$. Then, by using the identity
\begin{equation} \label{eq:EtaMoment}
E[Y^\eta] = \int_0^\infty \eta t^{\eta-1} \overline{F}(t) dt
\end{equation}
we obtain
\begin{align}
E \left[ e^{\theta Y} 1_{(Y \leq y)} \right] &= \overline{F}(0) - e^{\theta y} \overline{F}(y) + \theta \int_0^y e^{\theta t} \overline{F}(t) \, dt \notag \\
&\leq 1 + \theta \int_0^{1/\theta} \overline{F}(t) \, dt + \theta \int_0^{1/\theta} (e^{\theta t} - 1) \overline{F}(t) \, dt +  \theta \int_{1/\theta}^y e^{\theta t} \overline{F}(t) \, dt \notag \\
&\leq 1 + \theta E[Y] + e \theta^2 \int_0^{1/\theta} t \overline{F}(t) \, dt +  \theta \int_{1/\theta}^y e^{\theta t} \overline{F}(t) \, dt \notag \\
&\leq 1 + \theta E[Y] + \frac{e \theta^{ 2\wedge \beta}}{2 \wedge \beta} E[Y^{2\wedge \beta}] +  \theta \int_{1/\theta}^y e^{\theta t} \overline{F}(t) \, dt ,  \label{eq:SecondOrd}
\end{align}
where in the second inequality we use $e^x - 1 \leq x e^x$, $x \geq 0$, and in the last inequality we use $t^{2 - (2\wedge \beta)} \leq \theta^{-2 + (2\wedge \beta)}$ and \eqref{eq:EtaMoment} with $\eta = 2 \wedge \beta$.  Similarly, if $0 < \beta \leq 1$, then
\begin{align}
E \left[ e^{\theta Y} 1_{(Y \leq y)} \right] &\leq 1 + \theta \int_0^{1/\theta} e^{\theta t} \overline{F}(t) \, dt + \theta \int_{1/\theta}^y e^{\theta t} \overline{F}(t) \, dt \notag \\
&\leq 1 +  \frac{e \theta^{\beta}}{\beta} E[Y^\beta] +  \theta \int_{1/\theta}^y e^{\theta t} \overline{F}(t) \, dt . \label{eq:FirstOrd}
\end{align}
Next, by Markov's inequality we have
$$\overline{F}(t) \leq E[Y^\beta] t^{-\beta},$$
which, in combination with \eqref{eq:SecondOrd} and \eqref{eq:FirstOrd}, gives 
\begin{equation}  \label{eq:Cases}
E \left[ e^{\theta Y} 1_{(Y \leq y)} \right] \leq \begin{cases}
1 + \theta E[Y] + \frac{e \theta^{2}}{2} E[Y^{2}] + E[Y^\beta] \theta \int_{1/\theta}^y e^{\theta t} t^{-\beta} dt , & \beta > 2, \\
1 + \theta E[Y] + \frac{e \theta^\beta}{\beta} E[Y^\beta] + E[Y^\beta] \theta \int_{1/\theta}^y e^{\theta t} t^{-\beta} dt , & 1 < \beta \leq 2, \\
1 + \frac{e \theta^\beta}{\beta} E[Y^\beta] +   E[Y^\beta] \theta \int_{1/\theta}^y e^{\theta t} t^{-\beta} dt, & 0 < \beta \leq 1.
\end{cases} 
\end{equation}
To analyze the remaining integral we split it as follows, for $\beta > 0$, 
\begin{align*}
\theta \int_{1/\theta}^y e^{\theta t} t^{-\beta} dt &\leq \theta^{1+\beta} \int_{1/\theta}^{y/2} e^{\theta t} dt + \theta \int_{y/2}^y e^{\theta t} t^{-\beta} \, dt \\
&\leq \theta^\beta e^{\theta y/2} + \theta y^{1-\beta} \int_{1/2}^1 e^{\theta y u} u^{-\beta} \, du \\
&\leq \theta^\beta e^{\theta y/2} + \theta y^{1-\beta} 2^{\beta} \int_{1/2}^1 e^{\theta y u} \, du \\
&\leq \theta^\beta e^{\theta y/2} +  2^{\beta} e^{\theta y}  y^{-\beta} ,
\end{align*}
from where it follows that 
\begin{align*}
&2e \theta^\beta E[Y^\beta] + E[Y^\beta] \theta \int_{1/\theta}^y e^{\theta t} t^{-\beta} dt  \\
&\leq 2e \theta^\beta E[Y^\beta] + E[Y^\beta] \theta^\beta e^{\theta y/2} + E[Y^\beta]  2^{\beta} e^{\theta y}  y^{-\beta} \\
&\leq 2^{\beta} E[Y^\beta] e^{\theta y}  y^{-\beta} \left( 1 + e 2^{1-\beta} (\theta y)^\beta e^{-\theta y} + 2^{-\beta} (\theta y)^\beta e^{-\theta y/2}  \right) \\
&\leq 2^{\beta} E[Y^\beta] e^{\theta y}  y^{-\beta} \left( 1 + 2e \sup_{t \geq 1} t^\beta e^{-t} + \sup_{t \geq 1/2} t^{\beta} e^{-t}  \right).
\end{align*}
Hence, we have shown that
$$2e \theta^\beta E[Y^\beta] + E[Y^\beta] \theta \int_{1/\theta}^y e^{\theta t} t^{-\beta} dt \leq K E[Y^\beta] e^{\theta y} y^{-\beta},$$
where $K = 2^\beta  \left(1 + (2e+1) \sup_{t \geq 1/2} t^\beta e^{-t} \right)$ does not depend on $\theta$ or the distribution of $Y$. Replacing the preceding inequality in \eqref{eq:Cases} and using $1+t \leq e^t$ give,
\begin{equation} \label{eq:TripleBound}
e^{-\theta x} E\left[ e^{\theta Y} 1_{(Y \leq y)} \right]^k \leq \begin{cases}
e^{-\theta (x - k E[Y]) + e k \theta^2 E[Y^2] + K k E[Y^\beta] e^{\theta y} y^{-\beta} } , & \beta > 2, \\
e^{-\theta (x - k E[Y])  + K k E[Y^\beta] e^{\theta y} y^{-\beta} } , & 1 < \beta \leq 2, \\
e^{-\theta x + K k E[Y^\beta] e^{\theta y} y^{-\beta}}, & 0 < \beta \leq 1. 
\end{cases}
\end{equation}
Now, to complete the proof, we optimize the choice of $\theta$ in the preceding bounds. For $0 < \beta < 1$, choose $\theta = \frac{1}{y} \log \left( \frac{x}{K k E[Y^\beta] y^{1-\beta}} \right)$ and note that for all $1 \leq k \leq x^\beta/E[Y^\beta]$ and $x \geq e^{(Ke)^{1/(1-\beta)}}$,
$$\theta y \geq \log\left( \frac{(\log x)^{1-\beta}}{K} \right) \geq 1.$$ 
Then,
\begin{align*}
e^{-\theta x + K k E[Y^\beta] e^{\theta y} y^{-\beta}} &= e^{-(\log x) \log \left( \frac{ x^\beta (\log x)^{1-\beta}}{K e k E[Y^\beta]} \right)   } \\
&\leq  e^{-(\log x) \log \left( \frac{(\log x)^{1-\beta}}{K e} \right)  } \\
&= e^{-(1-\beta)(\log x)(\log\log x) \left( 1 - \frac{\log (eK)}{(1-\beta)\log\log x} \right) } .
\end{align*}

Now, for $\beta > 1$, set $\theta = \frac{1}{y} \log\left( \frac{(x-kE[Y]) y^{\beta-1}}{K x} \right)$ and note that for and $x \geq e \vee (Ke/\epsilon)^{2/(\beta-1)}$, 
$$\theta y \geq \log\left( \frac{\epsilon y^{\beta-1}}{K } \right) \geq \log\left( \frac{\epsilon x^{(\beta-1)/2}}{K } \right) \geq 1.$$
Then, for $1 < \beta \leq 2$ and  all $1 \leq k \leq (1-\epsilon) x/(E[Y] \vee E[Y^\beta])$,
\begin{align*}
e^{-\theta (x-kE[Y]) + K k E[Y^\beta] e^{\theta y} y^{-\beta}} &= e^{- \frac{(x-kE[Y])}{y} \log\left( \frac{(x-kE[Y]) y^{\beta-1}}{Kx} \right) + k E[Y^\beta]  \frac{(x-kE[Y])}{xy}  }\\
&\leq e^{- \frac{(x-kE[Y])}{y} \log\left( \frac{\epsilon y^{\beta-1}}{Ke } \right) }  \\
&\leq e^{-\epsilon (\beta-1) (\log x)^2 \left( 1 - \frac{\log\log x}{\log x} - \frac{\log(Ke/\epsilon)}{(\beta-1)\log x}   \right) }  .
\end{align*}
In addition, for $\beta > 2$ note that 
$$\sup_{x \geq e} e k \theta^2 E[Y^2] \leq \sup_{x\geq e} \frac{e x}{y^2} \left( \log\left( \frac{ y^{\beta-1}}{K} \right) \right)^2 \leq \sup_{x \geq e} \frac{e (\beta-1)^2 (\log x)^4}{x} \leq e^5 (\beta-1)^2.$$
Finally, by combining the preceding two bounds with the first two inequalities in \eqref{eq:TripleBound}, we derive
\begin{align*}
P\left( \sum_{i=1}^k Y_i > x, \, \max_{1\leq i \leq k} Y_i \leq y  \right) &\leq e^{-\epsilon (\beta-1) (\log x)^2 \left( 1 - \frac{\log\log x}{\log x} - \frac{\log(Ke/\epsilon)}{(\beta-1)\log x} \right) + e^5 (\beta-1)^2}
\end{align*}
 for any $\beta > 1$. 
\qed\end{proof}

\bigskip

As an immediate corollary to the preceding lemma we obtain:

\begin{cor} \label{C.SimpleTrunc}
Suppose that $Y_1, Y_2, \dots$ are nonnegative iid random variables with the same distribution as $Y$, where $E[Y^\beta] < \infty$ for some $\beta > 0$. Then, for any $\kappa > 0$ there exists a constant $x_0 > 0$ that does not depend on the distribution of $Y$ such that
$$\sup_{1 \leq k \leq m_\beta(x)} P\left( \sum_{i=1}^k Y_i > x, \, \max_{1\leq i \leq k} Y_i \leq x/\log x  \right) \leq x^{-\kappa}$$
for all $x \geq x_0$, where
$$m_\beta(x) = \begin{cases}
\frac{x^\beta}{E[Y^\beta]}, & 0 < \beta < 1, \\
\frac{(1-\epsilon) x}{E[Y] \vee E[Y^\beta]}, & \beta > 1, 0 < \epsilon < 1. 
\end{cases}$$
\end{cor}

\bigskip

Lemma \ref{L.Bound1Iter} below gives a bound for the distribution of $W_{n+1}$ in terms of that of $W_{n}$. This lemma can also be used to prove the corresponding uniform bound for $W_n$ in the case when $Q$ dominates recursion \eqref{eq:GeneralPR}. In the statement of the lemma we assume that $1/L(x)$ is locally bounded on $[1, \infty)$. 

\begin{lem} \label{L.Bound1Iter}
Suppose that $P(N > x) \leq x^{-\alpha} L(x)$, with $\alpha > 1$ and $L(\cdot)$ slowly varying, and \linebreak $E[N] \max\{E[C^{\alpha}], E[C]\} < \eta < 1$.  Then, for any $c > 0$, $0 < \epsilon < 1$, and $0 < \delta < 1 \wedge (\alpha-1)/2$,  there exist constants $K = K(\delta,\epsilon,c,\eta) > 0$ and $x_0 = x_0(\delta,\epsilon,c,\eta) > 0$, that do not depend on $n$, such that for all $1 \leq n \leq c\log x/|\log\eta|$ and all $x \geq x_0$,
\begin{equation*} 
P(W_{n+1} > x) \leq K \eta^{(2\wedge (\alpha-\delta)) n} x^{-\alpha} L(x) + E[N] P(C W_{n} > (1-\epsilon) x) ,
\end{equation*}
where $C$ and $W_n$ are independent.
\end{lem}

{\sc Remark:} Note that the condition $E[N] \max\{E[C^\alpha], E[C]\} < 1$ is natural since it is needed for the finiteness of $E[R^\beta]$ for any $\beta < \alpha$. It is also in agreement with Lemma \ref{L.Finite_n} in the sense that it is a necessary condition for the convergence (as $n \to \infty$) of the sum appearing in \eqref{eq:Asymp_R_n}. The choice of $\eta$ is also suggested by the fact that for $\beta < \alpha$ one can obtain a weaker uniform bound by applying the moment estimate on $E[W_n^\beta]$ from Lemma \ref{L.GeneralMoment}, i.e., $P(W_n > x) \leq E[W_n^\beta] x^{-\beta} \leq K_\beta (E[N]\max\{ E[C], E[C^\beta] \})^n x^{-\beta}$.

Before going into the proof, we would like to emphasize that special care goes into making sure that $K$ and $x_0$ in the statement of the lemma do not depend on $n$. This is important since Lemma~\ref{L.Bound1Iter} will be applied iteratively in the proof of Proposition \ref{P.UniformBound}, where one does not want $K$ and $x_0$ to grow from one iteration to the next.

\begin{proof}[Proof of Lemma \ref{L.Bound1Iter}]
By convexity of $f(\theta) = E[C^\theta]$, $\max\{E[C^\alpha], E[C]\} \geq \max\{ E[C^{\alpha-\delta}], E[C] \}$, implying
$$\varepsilon \triangleq \frac{\eta}{E[N] \max\{E[C^{\alpha-\delta}], E[C]\}} - 1 > 0.$$ 
Next, recall that $W_{n+1} \stackrel{\mathcal{D}}{=} \sum_{i=1}^N C_i W_{n,i}$ where $W_{n,i}$ are iid copies of $W_n$, let $Y \stackrel{\mathcal{D}}{=} Y_i = C_i W_{n,i}$ and $\beta = \alpha- \delta > 1$. Note that by Lemma \ref{L.GeneralMoment} there exists a constant $K_1 > 0$ (that does not depend on $n$) such that, 
\begin{align}
E[Y^\beta] &= E[C^\beta] E[W_n^\beta] \notag \\
&\leq K_1 (E[N] \max\{E[C^{\alpha-\delta}], E[C]\})^n \notag  \\
&= K_1 (1+\varepsilon)^{-n} \eta^n, \label{eq:NewMomentBound}
\end{align}
where the last equality comes from the definition of $\varepsilon$. And since $E[Y] = E[Q] (E[N] E[C])^n$ \linebreak $\leq E[Q] (E[N] \max\{ E[C^{\alpha-\delta}], E[C]\})^n$, then
\begin{equation} \label{eq:OtherMomentBound}
E[Y^\beta] \vee E[Y] \leq K_2 (1+\varepsilon)^{-n} \eta^{n}
\end{equation}
for some constant $K_2 > 0$ that does not depend on $n$.  With the intent of applying Corollary \ref{C.SimpleTrunc}, we define
$$y \triangleq \epsilon x \qquad \text{and} \qquad m_\beta(x) \triangleq \lfloor \epsilon^2 x/(E[Y^\beta] \vee E[Y]) \rfloor.$$
Let $M_k^{(i)}$ is the $i$th order statistic of $\{Y_1, \dots, Y_k \}$, with $M_k^{(k)}$ being the largest. Then,
\begin{align}
P\left( W_{n+1} > x \right) &= P\left( \sum_{i=1}^N Y_i > x \right) \notag \\
&\leq P\left( \sum_{i=1}^N Y_i > x, \, N \leq m_\beta(x) \right) + P\left( N > m_\beta(x) \right) \notag \\
&\leq P\left( \sum_{i=1}^N Y_i > x, \, M_N^{(N)} \leq (1-\epsilon) x, \, N \leq m_\beta(x) \right) \notag \\
&\hspace{5mm} + P\left( M_N^{(N)} > (1-\epsilon) x, \, N \leq m_\beta(x) \right) + P\left( N > m_\beta(x) \right) \notag \\
&\leq P\left( \sum_{i=1}^N Y_i > x, \, M_N^{(N)} \leq (1-\epsilon) x, \, M_N^{(N-1)} \leq y/\log y, \, N \leq m_\beta(x) \right) \label{eq:FirstIneq} \\
&\hspace{5mm} + P\left( M_N^{(N-1)} > y /\log y, \, N \leq m_\beta(x) \right) \label{eq:SecondIneq} \\
&\hspace{5mm}  + P\left( M_N^{(N)} > (1-\epsilon) x, \, N \leq m_\beta(x) \right) + P\left(N > m_\beta(x)\right) . \label{eq:ThirdIneq}
\end{align}
Note that the term in \eqref{eq:FirstIneq} can be bounded as follows
\begin{align*}
&P\left( \sum_{i=1}^N Y_i > x, \, M_N^{(N)} \leq (1-\epsilon) x, \, M_N^{(N-1)} \leq y/\log y, \, N \leq m_\beta(x) \right) \\
&\leq P\left( \sum_{i=1}^N Y_i - M_N^{(N)} > y, \, M_N^{(N-1)} \leq y/\log y, \, N \leq m_\beta(x) \right) \\
&\leq P\left( \sum_{i=1}^{N} Y_i > y, \, M_{N}^{(N)} \leq y/\log y, \, N \leq m_\beta(x) \right) \\
&\leq P\left( \sum_{i=1}^{m_\beta(x)} Y_i > y, \, \max_{1\leq i < m_\beta(x)} Y_i \leq y/\log y \right).
\end{align*}
Fix $\nu = \alpha + \delta + c(\alpha-\delta)$, then, by Corollary \ref{C.SimpleTrunc}, there exists a constant $x_1 \geq e$, that does not depend on the distribution of $Y$ (and therefore, does not depend on $n$), such that
\begin{align*}
P\left( \sum_{i=1}^{m_\beta(x)} Y_i > y, \, \max_{1\leq i < m_\beta(x)} Y_i \leq y/\log y \right) &\leq y^{-\nu} = \epsilon^{-\nu} \eta^{\frac{c(\alpha-\delta)}{|\log\eta|} \cdot \log x} x^{-\alpha-\delta} \\
&\leq \epsilon^{-\nu} \eta^{(\alpha-\delta) n} x^{-\alpha-\delta} = \epsilon^{-\nu} \frac{x^{-\delta}}{L(x)} \,\eta^{\beta n} x^{-\alpha} L(x) \\
&\leq \epsilon^{-\nu} \sup_{t \geq 1} \frac{t^{-\delta}}{L(t)} \, \eta^{\beta n} x^{-\alpha} L(x)
\end{align*} 
for all $y \geq x_1$, where the second inequality follows from the assumption $n \leq c\log x/|\log\eta|$, and in the second equality we use the definition $\beta = \alpha-\delta$.  To bound \eqref{eq:SecondIneq}, we condition on $N$,
\begin{align*}
P\left( M_N^{(N-1)} > y /\log y, \, N \leq m_\beta(x) \right)  &= \sum_{k=1}^{m_\beta(x)} P\left( M_k^{(k-1)} > y /\log y  \right) P(N = k) \\
&\leq \sum_{k=1}^{m_\beta(x)} \binom{k}{2} P(Y > y/\log y)^2 P(N = k) \\
&\leq E\left[ N^2 1_{(N \leq m_\beta(x))} \right] P(Y > y/\log y)^2 \\
&\leq E\left[  N^{2 \wedge \beta} \right] m_\beta(x)^{(2-\beta)^+} P(Y > y/\log y)^2 ,
\end{align*}
where in the last inequality we use $N \leq m_\beta(x)$ in case $N$ does not have a second moment.  Now, by Markov's inequality and the definition of $m_\beta(x)$, 
\begin{align*}
m_\beta(x)^{(2-\beta)^+} P(Y > y/\log y)^2 &\leq m_\beta(x)^{(2-\beta)^+} \left( \frac{E[Y^{\beta}] (\log y)^\beta}{y^\beta}  \right)^2 \\
&\leq \left( \frac{ E[Y^\beta] }{E[Y^\beta] \vee E[Y]} \right)^{(2-\beta)^+} \frac{\epsilon^{(2-\beta)^+} E[Y^\beta]^{2 \wedge \beta} (\log y)^{2\beta} }{y^{2\beta \wedge (3\beta-2)}}  \\
&\leq \frac{\epsilon^{(2-\beta)^+} E[Y^\beta]^{2 \wedge \beta} (\log y)^{2\beta} }{y^{2\beta \wedge (3\beta-2)}} \\
&\leq \frac{\epsilon^{(2-\beta)^+} (K_1 (1+\varepsilon)^{-n} \eta^{n})^{2 \wedge \beta} (\log y)^{2\beta} }{y^{2\beta \wedge (3\beta-2)}} \qquad \text{(by \eqref{eq:NewMomentBound})}.
\end{align*}
Our choice of $\delta$ guarantees that $2\beta \wedge (3\beta -2) > \alpha + \delta$ and $\beta = \alpha-\delta > 1$, and therefore,
\begin{align*}
P\left( M_N^{(N-1)} > y /\log y, \, N \leq m_\beta(x) \right) &\leq K_3 \, \frac{\eta^{(2\wedge\beta)n}}{(1+\varepsilon)^{(2\wedge \beta)n}}  x^{-\alpha-\delta} \\
&\leq K_3 \, \frac{x^{-\delta}}{L(x)} \eta^{(2\wedge\beta)n}  x^{-\alpha} L(x) \\
&\leq K_3 \sup_{t \geq 1} \frac{t^{-\delta}}{L(t)} \, \eta^{(2\wedge\beta)n} x^{-\alpha} L(x)
\end{align*}
for all $x \geq x_2 = \epsilon^{-1} e$, where
$$K_3 = K_3(\epsilon,\delta) =  E\left[  N^{2 \wedge \beta} \right] \epsilon^{(2-\beta)^+-\alpha-\delta} K_1^{2\wedge \beta} \,  \sup_{t \geq e} \frac{(\log t)^{2\beta} }{t^{2\beta \wedge (3\beta-2) -\alpha-\delta}}.$$
To bound the second term in \eqref{eq:ThirdIneq}, we first note that by Potter's Theorem (see Theorem 1.5.6 (ii) on p. 25 in \cite{BiGoTe1987}), there exists a constant $x_3 = x_3(\varepsilon, \delta)$ such that for all $x \geq x_3$
\begin{align*}
P(N > m_\beta(x)) &\leq \frac{(m_\beta(x))^{-\alpha} L(m_\beta(x))}{x^{-\alpha} L(x)} \cdot x^{-\alpha} L(x) \\
&\leq (1+\varepsilon) \max\left\{ \left( \frac{m_\beta(x)}{x} \right)^{-\alpha+\delta}, \, \left( \frac{m_\beta(x)}{x} \right)^{-\alpha-\delta} \right\} x^{-\alpha} L(x) \\
&= (1+\varepsilon) \max\left\{ \left( \frac{E[Y^\beta] \vee E[Y]}{\epsilon^2} \right)^{\alpha-\delta}, \, \left( \frac{E[Y^\beta] \vee E[Y]}{\epsilon^2} \right)^{\alpha+\delta} \right\} x^{-\alpha} L(x) \\
&\leq \frac{(1+\varepsilon)}{\epsilon^{2(\alpha+\delta)}} (E[Y^\beta] \vee E[Y])^\beta x^{-\alpha} L(x) \\
&\leq \frac{K_2^\beta }{\epsilon^{2(\alpha+\delta)}} \cdot \frac{\eta^{\beta n}}{(1+\varepsilon)^{\beta n - 1}}  \cdot x^{-\alpha} L(x) \qquad \text{(by \eqref{eq:OtherMomentBound})} \\
&\leq K_4 \eta^{\beta n} x^{-\alpha} L(x).
\end{align*}
Finally, for the first term in \eqref{eq:ThirdIneq}, 
\begin{align*}
P\left( M_N^{(N)} > (1-\epsilon) x, \, N \leq m_\beta(x) \right) &\leq P\left( M_N^{(N)} > (1-\epsilon) x \right) \\
&\leq E[N] P(Y > (1-\epsilon) x).
\end{align*}
Combining the preceding bounds for \eqref{eq:FirstIneq} - \eqref{eq:ThirdIneq} and setting $x_0 = \max\{x_1, x_2, x_3\}$ and $K = (\epsilon^{-\nu} + K_3) \sup_{t \geq 1} \frac{t^{-\delta}}{L(t)} + K_4$ completes the proof. 
\qed\end{proof}

\bigskip

Finally, we give the proof of Proposition \ref{P.UniformBound}, the main technical contribution of Section \ref{S.NDominates}. 

\begin{proof}[Proof of Proposition \ref{P.UniformBound}]
Note that it is enough to prove the proposition for all $x \geq x_0$ for some $x_0 = x_0(\eta,\nu) > 1$, since for all $1 \leq x \leq x_0$ and $n \geq 1$, 
\begin{align*}
P(W_n > x) &= \frac{P(W_n > x)}{ \eta^n P(N > x)} \, \eta^n P(N > x)\\
&\leq \frac{E[Q] (E[N]E[C])^n x^{-1}}{\eta^n P(N > x)} \, \eta^n P(N > x) \qquad \text{(by Markov's inequality)} \\
&\leq \sup_{1 \leq t \leq x_0} \frac{E[Q]}{t P(N > t)} \, \cdot \eta^n P(N > x) .
\end{align*}
Next, choose $0 < \epsilon < 1$ such that
\begin{equation} \label{eq:EpsilonChoice}
E[N] E[C^\alpha]  \left( (1-\epsilon)^{-\alpha-1} + 2\epsilon  \right) \leq \eta,
\end{equation}
define $c = \nu/2$,
$$\gamma = \frac{1}{|\log\eta|} \log\left( \frac{\eta}{E[N] \max\{E[C^{\alpha}], E[C]\}} \right),$$
and select $0 < \delta < \min\{1, (\alpha-1)/2, c\gamma \}$. Now, by Lemma \ref{L.Bound1Iter}, there exist constants $K_1, x_1 > 0$ (that do not depend on $n$) such that
$$P(W_{n+1} > x) \leq K_1 \eta^{(2\wedge (\alpha-\delta)) n} P(N > x) + E[N] P(C W_{n} > (1-\epsilon) x)$$
for all $x \geq x_1$. Hence, by defining $n_0 = (2 \wedge(\alpha-\delta)-1)^{-1} (\log\eta)^{-1} \log(\epsilon E[N]E[C^\alpha])$, we obtain
\begin{equation} \label{eq:oneIter}
P(W_{n+1} > x) \leq K_1 E[N]E[C^\alpha] \epsilon \eta^{n} P(N > x) + E[N] P(C W_{n} > (1-\epsilon) x)
\end{equation}
for all $n \geq n_0$, and all $x \geq x_1$. 

Next, in order to derive an explicit bound for $P(W_n > x)$, we need the following two estimates \eqref{eq:boundForCN} and \eqref{eq:boundForC}. In this regard, choose $x_0 \geq 1 \vee x_1$ such that
\begin{equation} \label{eq:boundForCN}
P(C N > (1-\epsilon) x) \leq E[C^\alpha] (1-\epsilon)^{-\alpha-1} P(N > x)
\end{equation}
for all $x \geq x_0$. This is possible since by Lemma 4.2 from \cite{Jess_Miko_06} $P(C N > (1-\epsilon) x) \sim E[C^\alpha] (1-\epsilon)^{-\alpha} P(N > x)$. Also, by Markov's inequality, we have that for all $1\leq n \leq c\log x/|\log\eta|$, 
\begin{align}
P(C > (1-\epsilon) x/x_0) &\leq E[C^{\alpha+\nu}] (1-\epsilon)^{-\alpha-\nu} x_0^{\alpha+\nu}  x^{-\alpha-\nu}  \notag \\
&= \frac{E[C^{\alpha+\nu}]  x_0^{\alpha+\nu} }{(1-\epsilon)^{\alpha+\nu} x^{\nu/2} L(x)} x^{-\nu/2} P(N > x)  \notag \\
&\leq \frac{E[C^{\alpha+\nu}] x_0^{\alpha+\nu}}{(1-\epsilon)^{\alpha+\nu} x^{\nu/2}L(x)} \, \eta^n P(N > x)  ,\label{eq:boundForC}
\end{align} 
where in the second inequality we use $x^{-\nu/2} = x^{-c} = \eta^{\frac{c\log x}{|\log\eta|}} \leq \eta^n$. Now, define
$$K_2 = \max\left\{1, \, K_1, \, \sup_{x \geq x_0} \frac{E[C^{\alpha+\nu}] x_0^{\alpha+\nu}}{\epsilon E[C^\alpha] (1-\epsilon)^{\alpha+\nu} x^{\nu/2}L(x)} \right\}.$$

Now we proceed to derive bounds for $P(W_n > x)$ for different ranges of $n$. For all $1\leq n \leq n_0$ and all $x \geq x_0$, by Lemma \ref{L.Finite_n}, there exists a constant $K_0 \geq K_2$ such that
\begin{equation} \label{eq:InductionHyp}
P(W_n > x) \leq K_0 \,  \eta^n P(N > x).
\end{equation}

Next, for the values $n_0 \leq n \leq c \log x/|\log \eta|$ we proceed by induction using \eqref{eq:oneIter}.  To this end, suppose \eqref{eq:InductionHyp} holds for some $n$ in the specified range. Then, note that  by \eqref{eq:boundForC} and the induction hypothesis \eqref{eq:InductionHyp}, we have for all $x \geq x_0$,
\begin{align*}
P(C W_n > (1-\epsilon) x) &\leq P(C W_n > (1-\epsilon) x, C \leq (1-\epsilon)x/x_0) + P(C > (1-\epsilon)x/x_0) \\
&\leq \int_0^{(1-\epsilon)x/x_0} P(W_n > (1-\epsilon) x/y) P(C \in dy) + K_2 E[C^\alpha]  \epsilon \eta^n P(N > x) \\
&\leq K_0 \eta^n \int_0^{\infty} P(N > (1-\epsilon) x/y) P(C \in dy) + K_2 E[C^\alpha] \epsilon \eta^n P(N > x) \\
&= K_0 \eta^n P(CN > (1-\epsilon) x) + K_2 E[C^\alpha] \epsilon \eta^n P(N > x) \\
&\leq K_0 E[C^\alpha] \left( (1-\epsilon)^{-\alpha-1} + \epsilon \right) \eta^n P(N > x) ,
\end{align*}
where in the last inequality we used \eqref{eq:boundForCN} and $K_0 \geq K_2$. Then, by replacing the preceding bound in \eqref{eq:oneIter} and using \eqref{eq:EpsilonChoice}, we derive
\begin{align*}
P(W_{n+1} > x) &\leq K_0 E[N] E[C^\alpha]  \left( (1-\epsilon)^{-\alpha-1} + 2\epsilon  \right)  \eta^n P(N > x) \\
&\leq K_0 \eta^{n+1} P(N > x)
\end{align*}
for all $x \geq x_0$ and all $1 \leq n \leq c\log x/|\log\eta|$.  

Finally, for $n \geq c \log x/|\log\eta|$, we follow a different approach that comes from our moment estimates for $W_n$. Let
$$\varepsilon =  \frac{\eta}{E[N] \max\{E[C^{\alpha}], E[C]\}}  - 1 > 0$$
and note that by convexity
$$E[N] \max\{E[C^{\alpha-\delta}], E[C]\} \leq E[N] \max\{E[C^\alpha], E[C]\}  = (1+ \varepsilon)^{-1} \eta.$$ 
Then, by Markov's inequality and Lemma \ref{L.GeneralMoment}, we have
\begin{align}
P(W_n > x) &\leq E[W_n^{\alpha-\delta}] x^{-\alpha+\delta} \notag \\
&\leq K_{\alpha-\delta} (E[N] \max\{E[C^{\alpha-\delta}], E[C]\})^n x^{-\alpha+\delta} \notag \\
&= K_{\alpha-\delta} (1+\varepsilon)^{- n} \eta^n x^{-\alpha+\delta} \notag \\
&\leq K_{\alpha-\delta} x^{- \log(1+\varepsilon) c /|\log\eta|}  \eta^n x^{-\alpha+\delta} \label{eq:boundForW_n}
\end{align}
for all $x > 0$. Note that the preceding bound,
\begin{align*}
 \frac{\log (1+\varepsilon)}{|\log\eta|} &= \frac{1}{|\log\eta|} \log\left( \frac{\eta}{E[N] \max\{E[C^{\alpha}], E[C]\}} \right) =  \gamma,
\end{align*}
and \eqref{eq:boundForW_n} yield
\begin{align*}
P(W_n > x) &\leq K_{\alpha-\delta} \eta^n x^{-c\gamma -\alpha+\delta} \\
&\leq  K_{\alpha-\delta} \eta^n x^{-\alpha+\delta -c\gamma } = K_{\alpha-\delta} \eta^n \frac{x^{\delta -c\gamma }}{L(x)} P(N > x) \\
&\leq K_{\alpha-\delta} \sup_{t \geq 1} \frac{t^{\delta-c\gamma}}{ L(t)} \, \eta^n P(N > x) 
\end{align*}
for all $x \geq 1$; recall that $\delta < c\gamma$. Thus, setting $K = \max\{K_0, \, K_{\alpha-\delta} \sup_{t \geq 1} t^{\delta-c\gamma} (L(t))^{-1} \}$ completes the proof. 
\qed\end{proof}

\subsection{The case when $Q$ dominates} \label{SS.QDominate_Proofs}

We end the paper with the proof of Lemma \ref{L.Finite_nQ} (the proof of Lemma \ref{L.W_n_Finite_nQ} is basically the same) and a sketch of the proof of Proposition \ref{P.UniformBoundQ}. As mentioned before, the proofs of the other results presented in Section \ref{S.QDominates} have been omitted since they are very similar to those from Section \ref{S.NDominates}.

\begin{proof}[Proof of Lemma \ref{L.Finite_nQ}]
We proceed by induction in $n$. By Lemma 4.2 from \cite{Jess_Miko_06},
$$P(C Q > x) \sim E[C^\alpha] P(Q > x),$$
by Lemma 3.7(1) from the same source,
$$P\left( \sum_{i=1}^{N^{(0)}} C_i^{(1)} Q^{(1)}_i > x \right) \sim E[N] P(C Q > x) \sim E[N] E[C^\alpha] P(Q> x),$$
and by Lemma 3.1, again from the same source, we have
\begin{align*}
P(R^{(1)} > x) &= P\left( \sum_{i=1}^{N^{(0)}} C_i^{(1)} Q_i^{(1)} + Q^{(0)} > x \right) \\
&\sim P\left( \sum_{i=1}^{N^{(0)}} C_i^{(1)} Q_i^{(1)} > x \right) + P(Q > x) \\
&\sim (\rho_\alpha + 1) P(Q > x).
\end{align*}
Now suppose that we have
$$P(R^{(n)} > x) \sim \sum_{k=0}^n \rho_\alpha^k \, P(Q > x).$$
Then,
\begin{align*}
P(R^{(n+1)} > x) &= P\left( \sum_{i=1}^{N^{(0)}} C_i^{(1)} R_{i}^{(n)} + Q^{(0)} > x \right) \\
&\sim P\left( \sum_{i=1}^{N^{(0)}} C_i^{(1)} R_{i}^{(n)} > x \right) + P(Q > x) \\
&\sim E[N] E[C^\alpha] P(R^{(n)} > x) + P(Q > x) \\
&\sim \left( \rho_\alpha \sum_{k=0}^n \rho_\alpha^k + 1 \right) P(Q > x) \\
&= \sum_{k=0}^{n+1} \rho_\alpha^k \, P(Q > x).
\end{align*}
\qed\end{proof}

\begin{proof}[Sketch of the proof of Proposition \ref{P.UniformBoundQ}]
By Markov's inequality 
$$P(N > x) \leq E[N^{\alpha+\nu}] x^{-\alpha-\nu}$$
for all $x > 0$. Use Lemma \ref{L.Bound1Iter} to obtain 
$$P(W_{n+1} > x) \leq K_1 E[N] E[C^\alpha] \epsilon \eta^n P(Q > x) + E[N] P(C W_n > (1-\epsilon) x)$$
for all $n_0 \leq n \leq \kappa \log x$ and all $x \geq x_1$ (for suitably chosen constants $\epsilon, n_0, \kappa$). Choose $x_0 \geq 1 \vee x_1$ such that
$$P(CQ > (1-\epsilon) x) \leq E[C^\alpha] (1-\epsilon)^{-\alpha-1} P(Q > x).$$
The rest of the proof continues as in Proposition \ref{P.UniformBound} with some modifications. 
\qed\end{proof}


\ack

The authors are grateful to Professor Charles Goldie for pointing out a reference, and also to an anonymous referee for his or her helpful comments. 


\end{document}